\documentclass[english,british,reqno]{amsart}
\usepackage[T1]{fontenc}
\usepackage[latin9]{inputenc}
\usepackage[a4paper]{geometry}
\usepackage{amsthm}
\usepackage{amstext}
\usepackage{amssymb}
\usepackage{setspace}
\makeatletter
\numberwithin{equation}{section}
\numberwithin{figure}{section}
\usepackage{enumitem}		% customizable list environments
\theoremstyle{plain}
\newtheorem{thm}{\protect\theoremname}[section]
  \theoremstyle{plain}
  \newtheorem{cor}[thm]{\protect\corollaryname}
  \theoremstyle{remark}
  \newtheorem*{rem*}{\protect\remarkname}
  \theoremstyle{definition}
  \newtheorem{defn}[thm]{\protect\definitionname}
  \theoremstyle{plain}
  \newtheorem{prop}[thm]{\protect\propositionname}
  \theoremstyle{remark}
  \newtheorem{rem}[thm]{\protect\remarkname}
  \theoremstyle{definition}
  \newtheorem{example}[thm]{\protect\examplename}
  \theoremstyle{plain}
  \newtheorem{conjecture}[thm]{\protect\conjecturename}
  \theoremstyle{remark}
  \newtheorem*{acknowledgement*}{\protect\acknowledgementname}
  \theoremstyle{plain}
  \newtheorem{fact}[thm]{\protect\factname}
  \theoremstyle{plain}
  \newtheorem{lem}[thm]{\protect\lemmaname}

\usepackage{amscd}
\usepackage{mathptmx}
\usepackage{bbm}

\newcommand{\e}{\mathrm{e}}
\newcommand{\1}{\mathbbm{1}}
\newcommand{\N}{\mathbb{N}}
\newcommand{\F}{\mathbb{F}}
\newcommand{\Z}{\mathbb{Z}}

\newcommand{\R}{\mathbb{R}}

\renewcommand{\Pi}{\pi}
\renewcommand{\emptyset}{\varnothing}
\renewcommand{\hat}{\widehat}

\newcommand\id{\mathrm{id}}

\DeclareMathOperator*{\card}{card}

\usepackage{hyperref}
\@ifundefined{textmu}
 {\usepackage{textcomp}}{}

\makeatother

\usepackage{babel}
  \addto\captionsbritish{\renewcommand{\acknowledgementname}{Acknowledgement}}
  \addto\captionsbritish{\renewcommand{\conjecturename}{Conjecture}}
  \addto\captionsbritish{\renewcommand{\corollaryname}{Corollary}}
  \addto\captionsbritish{\renewcommand{\definitionname}{Definition}}
  \addto\captionsbritish{\renewcommand{\examplename}{Example}}
  \addto\captionsbritish{\renewcommand{\factname}{Fact}}
  \addto\captionsbritish{\renewcommand{\lemmaname}{Lemma}}
  \addto\captionsbritish{\renewcommand{\propositionname}{Proposition}}
  \addto\captionsbritish{\renewcommand{\remarkname}{Remark}}
  \addto\captionsbritish{\renewcommand{\theoremname}{Theorem}}
  \addto\captionsenglish{\renewcommand{\acknowledgementname}{Acknowledgement}}
  \addto\captionsenglish{\renewcommand{\conjecturename}{Conjecture}}
  \addto\captionsenglish{\renewcommand{\corollaryname}{Corollary}}
  \addto\captionsenglish{\renewcommand{\definitionname}{Definition}}
  \addto\captionsenglish{\renewcommand{\examplename}{Example}}
  \addto\captionsenglish{\renewcommand{\factname}{Fact}}
  \addto\captionsenglish{\renewcommand{\lemmaname}{Lemma}}
  \addto\captionsenglish{\renewcommand{\propositionname}{Proposition}}
  \addto\captionsenglish{\renewcommand{\remarkname}{Remark}}
  \addto\captionsenglish{\renewcommand{\theoremname}{Theorem}}
  \providecommand{\acknowledgementname}{Acknowledgement}
  \providecommand{\conjecturename}{Conjecture}
  \providecommand{\corollaryname}{Corollary}
  \providecommand{\definitionname}{Definition}
  \providecommand{\examplename}{Example}
  \providecommand{\factname}{Fact}
  \providecommand{\lemmaname}{Lemma}
  \providecommand{\propositionname}{Proposition}
  \providecommand{\remarkname}{Remark}
\providecommand{\theoremname}{Theorem}

\begin{document}

\title{Recurrence and Pressure for Group Extensions}
\begin{abstract}
We investigate the thermodynamic formalism for recurrent potentials
on group extensions of countable Markov shifts. Our main result characterises
recurrent potentials depending only on the base space, in terms of
the existence of a conservative product measure and a homomorphism
from the group into the multiplicative group of real numbers. We deduce
that, for a recurrent potential depending only on the base space,
the group is necessarily amenable. Moreover, we give equivalent conditions
for the base pressure and the skew product pressure to coincide. Finally,
we apply our results to analyse the Poincaré series of Kleinian groups
and the cogrowth of group presentations.
\end{abstract}

\author{Johannes Jaerisch}

\thanks{The author was supported by the research fellowship JA 2145/1-1 of
the German Research Foundation (DFG)}

\address{Department of Mathematics, Graduate School of Science, Osaka University,
1-1 Machikaneyama, Toyonaka, Osaka, 560-0043 Japan }

\email{jaerisch@cr.math.sci.osaka-u.ac.jp}

\urladdr{http://cr.math.sci.osaka-u.ac.jp/\textasciitilde{}jaerisch/}



\date{27th May 2013}

\maketitle

\section{Introduction and Statement of Results}

In this paper, we investigate the central ideas of the thermodynamic
formalism for recurrent potentials in the context of group extensions
of countable Markov shifts $\left(\Sigma,\sigma\right)$ with alphabet
$I\subset\N$ and left shift map $\sigma:\Sigma\rightarrow\Sigma$.
For a countable group $G$ and for a semigroup homomorphism $\Psi:I^{*}\rightarrow G$,
where $I^{*}$ denotes the free semigroup generated by $I$, the skew
product dynamical system $\sigma\rtimes\Psi$, which is given by 
\[
\sigma\rtimes\Psi:\Sigma\times G\rightarrow\Sigma\times G,\quad\left(\sigma\rtimes\Psi\right)\left(\omega,g\right)=\left(\sigma\left(\omega\right),g\Psi\left(\omega_{1}\right)\right),\,\,\omega=\left(\omega_{1},\omega_{2},\dots\right)\in\Sigma,\,\, g\in G,
\]
is called a \emph{group-extended Markov system} (see Section \ref{sub:Symbolic-Dynamics}
for details on Markov shifts). Let $\varphi:\Sigma\rightarrow\R$
denote a Hölder continuous potential, which extends to the potential
$\varphi\circ\pi_{1}:\Sigma\times G\rightarrow\R$, where $\pi_{1}:\Sigma\times G\rightarrow\Sigma$
denotes the canonical projection on $\Sigma$. The \emph{Perron-Frobenius
operator} (\cite{MR0234697,bowenequilibriumMR0442989}) is given by
$\mathcal{L}_{\varphi\circ\pi_{1}}\left(f\right)\left(x\right):=\sum_{\left(\sigma\rtimes\Psi\right)\left(y\right)=x}\e^{\varphi\circ\pi_{1}\left(y\right)}f\left(y\right)$.
The potential $\varphi\circ\pi_{1}$ is called \emph{recurrent} (\cite{MR1738951,MR1818392})
if there exists $\rho>0$ and a conservative measure $\nu$ such that
$\mathcal{L}_{\varphi\circ\pi_{1}}^{*}\left(\nu\right)=\rho\nu$.
Here, the measure is a $\sigma$-finite Borel measure on $\Sigma\times G$,
and $\mathcal{L}_{\varphi\circ\pi_{1}}^{*}\left(\nu\right)=\rho\nu$
means that, for every $f\in L^{1}\left(\nu\right)$, we have $\nu\left(\mathcal{L}_{\varphi\circ\pi_{1}}\left(f\right)\right)=\rho\nu\left(f\right)$.
In particular, the sum in the definition of $\mathcal{L}_{\varphi\circ\pi_{1}}\left(f\right)$
converges absolutely for $\nu$-almost every $x$. If $\varphi\circ\pi_{1}$
is recurrent, then the logarithm of $\rho$ is given by the \emph{Gurevi\v{c}
pressure} $\mathcal{P}\left(\varphi\circ\pi_{1},\sigma\rtimes\Psi\right)$
of $\varphi\circ\pi_{1}$ with respect to $\sigma\rtimes\Psi$. Moreover,
there exists a continuous function $h:\Sigma\times G\rightarrow\R^{+}$
such that $\mathcal{L}_{\varphi\circ\pi_{1}}\left(h\right)=\rho h$,
and the measure $h\, d\nu$ is the equilibrium measure of $\varphi\circ\pi_{1}$.
For details on these results of Sarig (\cite{MR1818392}), we refer
to Section \ref{sub:Recurrent-potentials}. That the equilibrium measure
is ergodic follows from work of Aaronson, Denker and Urba\'{n}ski
(\cite{MR1107025}). 

For a finitely primitive Markov shift $\Sigma$ (see Definition \ref{mixing-definitions})
and a Hölder continuous potential $f:\Sigma\rightarrow\R$ with finite
Gurevi\v{c} pressure, there exists a unique $\sigma$-invariant Gibbs
measure for $f$, which is denoted by $\mu_{f}$ (see Theorem \ref{thm:existence-of-gibbs-measures}).
Throughout, we let $\pi_{2}:\Sigma\times G\rightarrow G$ denote the
projection on $G$ and let $\lambda$ denote the Haar measure on $G$. 

We are now in the position to state our main result, which characterises
when $\varphi\circ\pi_{1}$ is recurrent.
\begin{thm}
\label{thm:main-theorem}Let $\Sigma$ be finitely primitive and let
$\left(\Sigma\times G,\sigma\rtimes\Psi\right)$ be a topologically
mixing group-extended Markov system. Let $\varphi:\Sigma\rightarrow\R$
be Hölder continuous with $\mathcal{P}\left(\varphi\circ\pi_{1},\sigma\rtimes\Psi\right)<\infty$.
Then $\varphi\circ\pi_{1}$ is recurrent if and only if there exists
a (unique) group homomorphism $c:G\rightarrow\left(\R^{+},\cdot\right)$
such that $\mathcal{P}\left(\varphi_{c},\sigma\right)<\infty$ and
\linebreak $\mu_{\varphi_{c}}\times\lambda$ is conservative, where
\textup{$\varphi_{c}$ }\textup{\emph{is given by}} $\varphi_{c}\left(x\right):=\varphi\left(x\right)-\log c\left(\Psi\left(x_{1}\right)\right)$
for $x=\left(x_{1},x_{2},\dots\right)\in\Sigma$. Moreover, if $\varphi\circ\pi_{1}$
is recurrent, then the following four statements hold. 
\begin{enumerate}
\item \textup{$\mathcal{P}\left(\varphi\circ\pi_{1},\sigma\rtimes\Psi\right)=\mathcal{P}\left(\varphi_{c},\sigma\right).$\label{enu:pressure-relation}}
\item The unique (up to a constant multiple) continuous function $h:\Sigma\times G\rightarrow\R^{+}$,
which is fixed by $\e^{-\mathcal{P}\left(\varphi\circ\pi_{1},\sigma\rtimes\Psi\right)}\mathcal{L}_{\varphi\circ\pi_{1}}$
and which is bounded on cylindrical sets, is given by $h=\left(h_{1}\circ\pi_{1}\right)\left(c\circ\pi_{2}\right)$,
where $h_{1}:\Sigma\rightarrow\R^{+}$ is the unique Hölder continuous
fixed point of $\e^{-\mathcal{P}\left(\varphi_{c},\sigma\right)}\mathcal{L}_{\varphi_{c}}$.\label{enu:eigenfunction}
\item The unique (up to a constant multiple) measure $\nu$ on $\Sigma\times G$,
which is fixed by \textup{$\e^{-\mathcal{P}\left(\varphi\circ\pi_{1},\sigma\rtimes\Psi\right)}\mathcal{L}_{\varphi\circ\pi_{1}}^{*}$
}and which is positive and finite on cylindrical sets, is given by
$\nu=\left(c\circ\pi_{2}\right)^{-1}\, d\left(\nu_{1}\times\lambda\right)$,
where $\nu_{1}$ is the unique probability measure fixed by $\e^{-\mathcal{P}\left(\varphi_{c},\sigma\right)}\mathcal{L}_{\varphi_{c}}^{*}.$\label{enu:conformal-measure}
\item $\mu_{\varphi_{c}}\times\lambda$ is the equilibrium measure of $\varphi\circ\pi_{1}$.
\label{enu:equilibrium-measure}
\end{enumerate}
\end{thm}
The first assertion of the following corollary is a generalisation
of a result of Kesten, who points out that, for an irreducible recurrent
random walks on a countable group, the group is necessarily amenable
(\cite[p.73]{MR0214137}). In order to state the second assertion
of the corollary, recall that $\varphi\circ\pi_{1}$ is said to be
\emph{positive recurrent} if the equilibrium measure is a finite measure
(see Definition \ref{definition-recurrence-pos-null} and Theorem
\ref{thm:sarig-recurrence-characterisation}). 
\begin{cor}
\label{cor:recurrence-implies-amenable}Under the hypotheses  of Theorem
\ref{thm:main-theorem}, we have the following.
\begin{enumerate}
\item \label{enu:recurrent-implies-amenable}If $\varphi\circ\pi_{1}$ is
recurrent, then $G$ is amenable. 
\item \label{enu:positive-recurrence-finite}$\varphi\circ\pi_{1}$ is positive
recurrent $\quad\Longleftrightarrow\quad$ $G$ is finite $\quad\Longleftrightarrow\quad$
$\left(\Sigma\times G,\sigma\rtimes\Psi\right)$ is finitely primitive. 
\end{enumerate}
\end{cor}
\begin{proof}
We start with the proof of (\ref{enu:recurrent-implies-amenable}).
By Theorem \ref{thm:main-theorem} (\ref{enu:equilibrium-measure}),
the ergodic equilibrium measure of $\varphi\circ\pi_{1}$ is the product
measure $\mu_{\varphi_{c}}\times\lambda$. Since the natural extension
of $\left(\Sigma,\sigma,\mu_{\varphi_{c}}\right)$ is an ergodic $\Z$-space,
it follows from \cite[Theorem 2.1]{MR0473096} that $\left(\Sigma,\sigma,\mu_{\varphi_{c}}\right)$
is an amenable $\Z$-space (see Zimmer \cite{MR0473096} for the definition).
Now, the ergodicity of $\mu_{\varphi_{c}}\times\lambda$ implies that
$G$ is amenable by \cite[Theorem 3.1]{MR0473096} (see also \cite{MR2104586}). 

Let us now prove (\ref{enu:positive-recurrence-finite}). If $\varphi\circ\pi_{1}$
is positive recurrent, then the equilibrium measure of $\varphi\circ\pi_{1}$
is finite. Hence, by Theorem \ref{thm:main-theorem} (\ref{enu:equilibrium-measure}),
the measure $\mu_{\varphi_{c}}\times\lambda$ is finite, which implies
that $G$ is finite. On the other hand, if $G$ is finite, then $\left(\Sigma\times G,\sigma\rtimes\Psi\right)$
is finitely primitive, which implies that $\varphi\circ\pi_{1}$ is
positive recurrent (see \cite[Corollary 2]{MR1955261}). The remaining
implication, that finite primitivity of $\left(\Sigma\times G,\sigma\rtimes\Psi\right)$
implies the finiteness of $G$, is straightforward and therefore omitted.\end{proof}
\begin{rem*}
By Theorem \ref{thm:main-theorem} (\ref{enu:equilibrium-measure}),
the equilibrium measure of a recurrent potential $\varphi\circ\pi_{1}$
satisfies the Gibbs property (\ref{eq:gibbs-equation}) with respect
to the potential $\varphi_{c}\circ\pi_{1}$.
\end{rem*}
As a further corollary of our main theorem, we obtain the following
equivalent conditions for the base pressure and the skew product pressure
of a recurrent potential $\varphi\circ\pi_{1}$ to coincide. 
\begin{cor}
\label{cor:base-skewproduct-coincide}Under the hypotheses of Theorem
\ref{thm:main-theorem}, suppose that $\varphi\circ\pi_{1}$ is recurrent.
Then the following statements are equivalent. 
\begin{enumerate}
\item \label{enu:c-is-trivial}The group homomorphism $c:G\rightarrow\left(\R^{+},\cdot\right)$
given by Theorem \ref{thm:main-theorem} is trivial. 
\item \label{enu:full-pressure}$\mathcal{P}\left(\varphi\circ\pi_{1},\sigma\rtimes\Psi\right)=\mathcal{P}\left(\varphi,\sigma\right)$.
\item \label{enu:eigenfunction-is-product} The function $h:\Sigma\times G\rightarrow\R^{+}$,
given by Theorem \ref{thm:main-theorem} (\ref{enu:eigenfunction}),
is independent of the second coordinate. 
\item \label{enu:eigenmeasure-is-product}The measure $\nu$ on $\Sigma\times G$,
given by Theorem \ref{thm:main-theorem} (\ref{enu:conformal-measure}),
is a product measure.
\item \label{enu:equilibrium-is-product}$\mu_{\varphi}\times\lambda$ is
the equilibrium measure of $\varphi\circ\pi_{1}$.  
\end{enumerate}
\end{cor}
\begin{proof}
If $c=1$ then we have $\mathcal{P}\left(\varphi,\sigma\right)=\mathcal{P}\left(\varphi_{c},\sigma\right)=\mathcal{P}\left(\varphi\circ\pi_{1},\sigma\rtimes\Psi\right)$
by Theorem \ref{thm:main-theorem} (\ref{enu:pressure-relation}),
which proves that (\ref{enu:c-is-trivial}) implies (\ref{enu:full-pressure}).
Next, we prove that (\ref{enu:full-pressure}) implies (\ref{enu:eigenfunction-is-product}).
Since $\Sigma$ is finitely primitive and $\varphi$ is Hölder continuous,
it follows from \cite[Theorem 2.4.3]{MR2003772} that there exists
a bounded Hölder continuous function $h_{1}:\Sigma\rightarrow\R^{+}$,
which is fixed by $\e^{-\mathcal{P}\left(\varphi,\sigma\right)}\mathcal{L}_{\varphi}$.
Consequently, we have that 
\[
\e^{-\mathcal{P}\left(\varphi,\sigma\right)}\mathcal{L}_{\varphi\circ\pi_{1}}\left(h_{1}\circ\pi_{1}\right)=\e^{-\mathcal{P}\left(\varphi,\sigma\right)}\left(\mathcal{L}_{\varphi}\left(h_{1}\right)\right)\circ\pi_{1}=h_{1}\circ\pi_{1}.
\]
Since $\mathcal{P}\left(\varphi,\sigma\right)=\mathcal{P}\left(\varphi\circ\pi_{1},\sigma\rtimes\Psi\right)$
by (\ref{enu:full-pressure}), we have $h=h_{1}\circ\pi_{1}$, which
proves (\ref{enu:eigenfunction-is-product}). The proof that (\ref{enu:c-is-trivial})
is equivalent to (\ref{enu:eigenfunction-is-product}) follows from
Theorem \ref{thm:main-theorem} (\ref{enu:eigenfunction}). The equivalence
of (\ref{enu:conformal-measure}) and (\ref{enu:eigenmeasure-is-product})
follows from Theorem \ref{thm:main-theorem} (\ref{enu:eigenfunction})
and (\ref{enu:conformal-measure}). To finish the proof, we verify
that (\ref{enu:c-is-trivial}) and (\ref{enu:equilibrium-is-product})
are equivalent. If $c=1$ then $\mu_{\varphi}\times\lambda$ is the
equilibrium measure of $\varphi\circ\pi_{1}$ by Theorem \ref{thm:main-theorem}
(\ref{enu:equilibrium-measure}). On the other hand, if $\mu_{\varphi}\times\lambda$
is the equilibrium measure of $\varphi\circ\pi_{1}$, then we have
that $\mu_{\varphi}\times\lambda$ is conservative. Hence, that $c$
is trivial follows from the uniqueness of $c$ in Theorem \ref{thm:main-theorem}. 
\end{proof}
In order to state the next proposition, we give the following definition.
For $k\in\N$, let $\Sigma^{k}$ denote the set of admissible words
of length $k$, and for $\omega\in\Sigma^{k}$, let \foreignlanguage{english}{$\left[\omega\right]$
denote} the cylindrical set given by $\omega$. 
\begin{defn}
\label{def:symmetric-on-average}Let $\left(\Sigma\times G,\sigma\rtimes\Psi\right)$
be an irreducible group-extended Markov system. Let $\varphi:\Sigma\rightarrow\R$
be Hölder continuous with $\mathcal{P}\left(\varphi\circ\pi_{1},\sigma\rtimes\Psi\right)<\infty$.
We say that $\varphi$ is \emph{symmetric on average} \emph{with respect
to $\Psi$} if 
\[
\sup_{g\in G}\limsup_{n\rightarrow\infty}\frac{\sum_{k=1}^{n}\e^{-k\mathcal{P}\left(\varphi\circ\pi_{1},\sigma\rtimes\Psi\right)}\sum_{\omega\in\Sigma^{k}:\Psi\left(\omega\right)=g}\e^{\sup S_{k}\varphi_{|\left[\omega\right]}}}{\sum_{k=1}^{n}\e^{-k\mathcal{P}\left(\varphi\circ\pi_{1},\sigma\rtimes\Psi\right)}\sum_{\omega\in\Sigma^{k}:\Psi\left(\omega\right)=g^{-1}}\e^{\sup S_{k}\varphi_{|\left[\omega\right]}}}<\infty.
\]
\end{defn}
\begin{prop}
\label{prop:fullpressure-symmetric-on-average}Under the hypotheses
of Theorem \ref{thm:main-theorem}, suppose that $\varphi\circ\pi_{1}$
is recurrent. Then we have that $\varphi$ is symmetric on average
with respect to $\Psi$ if and only if $\mathcal{P}\left(\varphi\circ\pi_{1},\sigma\rtimes\Psi\right)=\mathcal{P}\left(\varphi,\sigma\right)$.\end{prop}
\begin{rem*}
For a finite group $G$, every homomorphism $c:G\rightarrow\left(\R^{+},\cdot\right)$
is trivial. Furthermore, every Hölder continuous potential $\varphi\circ\pi_{1}:\Sigma\rightarrow\R$
with finite Gurevi\v{c} pressure is positive recurrent by Corollary
\ref{cor:recurrence-implies-amenable} (\ref{enu:positive-recurrence-finite}).
Hence, by combining the statement that (\ref{enu:c-is-trivial}) implies
(\ref{enu:full-pressure}) from Corollary \ref{cor:base-skewproduct-coincide}
and Proposition \ref{prop:fullpressure-symmetric-on-average}, we
have that $\varphi$ is symmetric on average with respect to $\Psi$. 
\end{rem*}
Since it will be convenient for some of our applications to investigate
irreducible and not necessarily aperiodic group extensions, we make
the following remark. 
\begin{rem}
\label{irreducible-remark} The results of Corollary \ref{cor:recurrence-implies-amenable}
also hold if $\left(\Sigma\times G,\sigma\rtimes\Psi\right)$ is irreducible
with period $p>1$.  Moreover, regarding  Proposition \ref{prop:fullpressure-symmetric-on-average},
we have $\mathcal{P}\left(\varphi\circ\pi_{1},\sigma\rtimes\Psi\right)=\mathcal{P}\left(\varphi,\sigma\right)$
if and only if 
\[
\sup_{g\in G_{0}}\limsup_{n\rightarrow\infty}\frac{\sum_{k=1}^{n}\e^{-kp\mathcal{P}\left(\varphi\circ\pi_{1},\sigma\rtimes\Psi\right)}\sum_{\omega\in\Sigma^{kp}:\Psi\left(\omega\right)=g}\e^{\sup S_{kp}\varphi_{|\left[\omega\right]}}}{\sum_{k=1}^{n}\e^{-kp\mathcal{P}\left(\varphi\circ\pi_{1},\sigma\rtimes\Psi\right)}\sum_{\omega\in\Sigma^{kp}:\Psi\left(\omega\right)=g^{-1}}\e^{\sup S_{kp}\varphi_{|\left[\omega\right]}}}<\infty,
\]
where $G_{0}:=\bigcup_{n\in\N}\Psi\left(\Sigma^{np}\right)$ is a
subgroup of $G$. 
\end{rem}
In the next example, which illustrates Proposition \ref{prop:fullpressure-symmetric-on-average},
we consider asymmetric nearest neighbour random walks on the additive
integers $\left(\Z,+\right)$. 
\begin{example}
\label{recurrent-non-symmetric-example}For the alphabet $I:=\left\{ 1,-1\right\} $
we consider the full shift $\Sigma:=I^{\N}$ and the group-extended
Markov system given by the canonical semigroup homomorphism $\Psi:I^{*}\rightarrow\left(\Z,+\right)$.
By Theorem \ref{thm:main-theorem}, each homomorphism $c:\left(\Z,+\right)\rightarrow\left(\R^{+},\cdot\right)$
gives rise to a recurrent potential $\varphi\circ\pi_{1}$, given
by $\varphi\left(x\right):=\log c\left(x_{1}\right)$, for which we
have $\mathcal{P}\left(\varphi\circ\pi_{1},\sigma\rtimes\Psi\right)=\mathcal{P}\left(0,\sigma\right)=\log2$.
Moreover, we have that 
\[
\mathcal{P}\left(\varphi\circ\pi_{1},\sigma\rtimes\Psi\right)=\log2\le\log(c(1)+c(-1))=\mathcal{P}\left(\varphi,\sigma\right),
\]
with equality if and only if $c=1$ (cf. Corollary \ref{cor:base-skewproduct-coincide}).
By Proposition \ref{prop:fullpressure-symmetric-on-average} we have
$c=1$ if and only if $\varphi$ is symmetric on average with respect
to $\Psi$. Also, note that the gap between $\mathcal{P}\left(\varphi\circ\pi_{1},\sigma\rtimes\Psi\right)$
and $\mathcal{P}\left(\varphi,\sigma\right)$ can be arbitrarily large.
Further, $\varphi\circ\pi_{1}$ is null recurrent by Corollary \ref{cor:recurrence-implies-amenable}
(\ref{enu:positive-recurrence-finite}).
\end{example}
Proposition \ref{prop:fullpressure-symmetric-on-average} extends
a result of Matsuzaki and Yabuki (Theorem \ref{thm:matsuzaki-yabuki})
for Kleinian groups to the general framework of the thermodynamic
formalism for group extensions of Markov shifts. In this way, we shed
new light on Theorem \ref{thm:matsuzaki-yabuki}. The result of Matsuzaki
and Yabuki was used in \cite{Jaerisch12a} to give a short new proof
of a lower bound of the exponent of convergence for normal subgroups
of Kleinian groups. By the results of Proposition \ref{prop:fullpressure-symmetric-on-average},
a similar lower bound can be proved in the setting of graph directed
Markov systems associated to free groups (\cite{Jaerisch12ab}). 

The proof of Theorem \ref{thm:main-theorem} relies on Sarig's thermodynamic
formalism for recurrent potentials (see \cite{MR1818392} or Theorem
\ref{thm:sarig-recurrence-characterisation}), which characterises
recurrence in terms of the existence of a conservative eigenmeasure
and a unique continuous eigenfunction of the Perron-Frobenius operator.
Sarig's construction of this eigenmeasure is similar to the construction
of the Patterson measure (\cite{MR0450547}). The work of Aaronson,
Denker and Urba\'{n}ski (\cite{MR1107025}), giving exactness and
pointwise dual ergodicity of eigenmeasures of the Perron-Frobenius
operator under certain assumptions, is also crucial for Sarig's results. 

The paper is organised as follows. In Section \ref{sec:Preliminaries}
we collect the necessary preliminaries from the thermodynamic formalism
for countable state Markov shifts, which includes the definition of
a Markov shift, Hölder continuous functions, the Gurevi\v{c} pressure
and Gibbs measures in Subsection \ref{sub:Symbolic-Dynamics}. In
Subsection \ref{sub:Recurrent-potentials} we give the definition
of recurrent potentials and a characterisation of these in terms of
fixed points of the Perron-Frobenius operator, which is due to Sarig.
In Subsection \ref{sub:Group-Extensions}, we review some recent results
on group-extended Markov systems and amenability. In Section \ref{sec:Proof-of-the}
we give the proofs of our main results.

Let us end this introductory section with the following subsections.
In Subsection \ref{sub:matsuzaki-yabuki} we show how to deduce Theorem
\ref{thm:matsuzaki-yabuki} of Matsuzaki and Yabuki for a large class
of Fuchsian groups by applying Proposition \ref{prop:fullpressure-symmetric-on-average}.
In Subsection \ref{sub:cogrowth}, we give an ergodic-theoretic proof
of an estimate on the cogrowth of group presentations. In Subsection
\ref{sub:Amenable-groups} we discuss how our main theorem relates
to amenability of groups.  In Subsection \ref{sub:Outlook} we briefly
relate our results to the classification of recurrent groups. We give
an extension of this classification to locally constant potentials
on irreducible group-extended Markov systems and we give a conjecture
for the general case of Hölder continuous potentials.

\subsection{A related result of Matsuzaki and Yabuki for Kleinian groups\label{sub:matsuzaki-yabuki}}

In order to state the result of Matsuzaki and Yabuki, recall that
a Kleinian group $G$ is a discrete group of isometries acting on
hyperbolic space, and that $G$ is defined to be of divergence type
if its Poincaré series diverges at the exponent of convergence $\delta\left(G\right)$.
We refer to \cite{Beardon,MR959135,MR1041575} for details on Kleinian
groups. 
\begin{thm}
[\cite{MR2486788}, Theorem 4.2, Corollary 4.3]  \label{thm:matsuzaki-yabuki}Let
$N$ denote a non-trivial normal subgroup of a Kleinian group $\Gamma$.
If $N$ is of divergence type, then we have that $\delta\left(N\right)=\delta\left(\Gamma\right)$. 
\end{thm}
In this subsection we briefly explain how Theorem \ref{thm:matsuzaki-yabuki}
is related to Proposition \ref{prop:fullpressure-symmetric-on-average}.
We consider the special case in which $\Gamma$ is a non-elementary
Fuchsian group of Schottky type acting on the Poincaré disc model
of hyperbolic $2$-space $\left(\mathbb{D},d\right)$, where $d$
denotes the hyperbolic metric. Recall that in this case, $\Gamma$
is isomorphic to the free group $\F_{t}=\left\langle g_{1},\dots,g_{t}\right\rangle $
with $t\ge2$. It is well-known (see e.g. \cite{MR662473}) that the
limit set of $\Gamma$ can be identified with the state space of the
Markov shift $\Sigma_{\F_{t}}$ with alphabet $I:=\left\{ g_{1},g_{1}^{-1},g_{2},g_{2}^{-1},\dots,g_{t},g_{t}^{-1}\right\} $
given by 
\[
\Sigma_{\F_{t}}=\{\omega\in I^{\N}:\omega_{i}\neq\omega_{i+1}^{-1}\}.
\]
Note that $\Sigma_{\F_{t}}$ is equal to the space of ends of the
Cayley graph of $\F_{t}$ with respect to $I$.

It is well-known that there exists a Hölder continuous potential $\zeta:\Sigma_{\F_{t}}\rightarrow\R$
which captures the modulus of the derivatives of the generators of
$\Gamma$ acting on the limit set (see \cite{MR662473}). This potential
has the geometric property that there exists $C>0$ such that, for
all $n\in\N$ and $\omega\in\Sigma_{\F_{t}}^{n}$, we have that 
\begin{equation}
C^{-1}\e^{\sup S_{n}\zeta_{|[\omega]}}\le\e^{-d(0,\omega_{1}\omega_{2}\dots\cdot\omega_{n}(0))}\le C\e^{\sup S_{n}\zeta_{|[\omega]}}.\label{eq:geometric-potential-vs-hyperbolicmetric}
\end{equation}
For a non-trivial normal subgroup $N$ of $\F_{t}$, let $\Psi_{N}:I^{*}\rightarrow\F_{t}/N$
be the unique semigroup homomorphism such that $\Psi_{N}\left(g_{i}\right):=Ng_{i}$,
for each $g_{i}\in I$. It follows from (\ref{eq:geometric-potential-vs-hyperbolicmetric})
that the exponents of convergence $\delta\left(\Gamma\right)$ and
$\delta\left(N\right)$ are characterised by the following equations,
which are often referred to as Bowen's formula (cf. \cite{MR556580}):
\begin{equation}
\mathcal{P}\left(\delta\left(\Gamma\right)\zeta,\sigma\right)=0,\quad\mathcal{P}\left(\delta\left(N\right)\left(\zeta\circ\pi_{1}\right),\sigma\rtimes\Psi_{N}\right)=0.\label{eq:poincareexponent-pressurezero}
\end{equation}
Moreover, using the fact that each $g\in\Gamma$ acts isometrically
on $\left(\mathbb{D},d\right)$, one immediately deduces from (\ref{eq:geometric-potential-vs-hyperbolicmetric})
that $\zeta$ is symmetric on average with respect to $\Psi_{N}$.
Moreover, we have that $N$ is of divergence type if and only if $\delta\left(N\right)\left(\zeta\circ\pi_{1}\right)$
is recurrent with respect to $\sigma\rtimes\Psi_{N}$. Hence, if $N$
is of divergence type, then Proposition \ref{prop:fullpressure-symmetric-on-average}
yields
\[
\mathcal{P}\left(\delta\left(N\right)\zeta,\sigma\right)=\mathcal{P}\left(\delta\left(N\right)\left(\zeta\circ\pi_{1}\right),\sigma\rtimes\Psi_{N}\right)=0,
\]
which, in light of (\ref{eq:poincareexponent-pressurezero}), gives
$\delta\left(N\right)=\delta\left(\Gamma\right)$.

\subsection{An application to the cogrowth of group presentations\label{sub:cogrowth}}

Let $\left\langle g_{1},\dots,g_{t}|r_{1},r_{2},\dots\right\rangle $,
$t\ge2$, be the presentation of a finitely generated group $G$ and
let $N$ denote a non-trivial normal subgroup of $\F_{t}$, generated
by the relations $r_{1},r_{2},\dots$. The \emph{cogrowth $\eta$}
of $\left\langle g_{1},\dots,g_{t}|r_{1},r_{2},\dots\right\rangle $,
which was independently introduced by Grigorchuk (\cite{MR599539})
and Cohen (\cite{MR678175}), is given by 
\[
\eta:=\frac{\log\gamma}{\log\left(2t-1\right)},\quad\mbox{where }\gamma:=\limsup_{n\rightarrow\infty}\left(\card\left(\left\{ \omega\in\F_{t}\cap N:\left|\omega\right|=n\right\} \right)\right)^{1/n}.
\]
In here, $\left|\omega\right|$ refers to the word length of $\omega$
with respect to $\left\{ g_{1},g_{1}^{-1},g_{2},g_{2}^{-1},\dots,g_{t},g_{t}^{-1}\right\} $.
It is known that $1/2<\eta\le1$ and that $\eta=1$ if and only if
$G$ is amenable (\cite{MR599539,MR678175}). Since $N$ is a normal
subgroup of $\F_{t}$, one easily verifies that 
\begin{equation}
\sum_{n\in\N}\left(\card\left\{ \omega\in N,\left|\omega\right|=n\right\} \right)\left(2t-1\right)^{-n/2}=\infty,\label{eq:cogrowth-recurrence}
\end{equation}
which immediately gives that $\left(2t-1\right)^{1/2}\le\gamma$ and
hence, $\eta\ge1/2$. We aim to give an ergodic-theoretic proof of
the strict inequality $\eta>1/2$ by using Proposition \ref{prop:fullpressure-symmetric-on-average}.
In order to apply Proposition \ref{prop:fullpressure-symmetric-on-average},
we consider the constant zero function $0:\Sigma_{\F_{t}}\rightarrow\R$
and the group-extended Markov system $\sigma\rtimes\Psi_{N}:\Sigma_{\F_{t}}\times\left(\F_{t}/N\right)\rightarrow\Sigma_{\F_{t}}\times\left(\F_{t}/N\right)$
as in Section \ref{sub:matsuzaki-yabuki}. Clearly, we have $\mathcal{P}\left(0,\sigma\right)=\log\left(2t-1\right)$
and $\mathcal{P}\left(0,\sigma\rtimes\Psi_{N}\right)=\log\gamma$.
Further, $0$ is symmetric on average with respect to $\Psi_{N}$.
By Proposition \ref{prop:fullpressure-symmetric-on-average} we conclude
that either $0$ is recurrent with respect to $\sigma\rtimes\Psi_{N}$
and $\gamma=2t-1$, or, $0$ is transient with respect to $\sigma\rtimes\Psi_{N}$.
In the latter case, it follows from (\ref{eq:cogrowth-recurrence})
that $\gamma\neq\left(2t-1\right)^{1/2}$. We have thus shown that
in both cases, $\left(2t-1\right)^{1/2}<\gamma$ and hence, $\eta>1/2$.

\subsection{Amenable groups\label{sub:Amenable-groups}}

Lifting a potential to a group extension may result in a drop of Gurevi\v{c}
pressure, that is $\mathcal{P}\left(\varphi\circ\pi_{1},\sigma\rtimes\Psi\right)<\mathcal{P}\left(\varphi,\sigma\right)$.
This phenomenon is intimately linked to non-amenability of the associated
group and was first observed by Kesten for random walks on countable
discrete groups driven by a symmetric independent identically distributed
process on a set of generators (\cite{MR0109367,MR0112053}). The
cogrowth criterion due to Grigorchuk and Cohen (\cite{MR599539,MR678175}),
which characterises amenability of groups in terms of cogrowth, is
a similar result, in which the independent identically distributed
process is replaced by a certain Markov process. Moreover, a result
of Brooks (\cite{MR783536}) for Kleinian groups certainly fits into
this context, where equality of the exponents of convergence of normal
subgroups of certain convex co-compact Kleinian groups is characterised
in terms of amenability. Recently, amenability of groups has been
studied in the framework of the thermodynamic formalism of group-extended
Markov systems, where the random walk is driven by a Gibbs measure
associated to a Hölder continuous potential on a Markov shift (\cite{JaerischDissertation11,Jaerisch11a,Stadlbauer11}).
We refer to Section \ref{sub:Group-Extensions} for a short review
of these results.

Let us now explain how these results are related to our results given
in this paper. Let $\left(\Sigma\times G,\sigma\rtimes\Psi\right)$
be an irreducible group-extended Markov system and let $\varphi:\Sigma\rightarrow\R$
be Hölder continuous.  Corollary \ref{cor:recurrence-implies-amenable}
(\ref{enu:recurrent-implies-amenable}) and Proposition \ref{prop:fullpressure-symmetric-on-average}
prove that, for a recurrent potential $\varphi\circ\pi_{1}$, the
associated group is amenable, and that the skew product pressure $\mathcal{P}\left(\varphi\circ\pi_{1},\sigma\rtimes\Psi\right)$
and the base pressure $\mathcal{P}\left(\varphi,\sigma\right)$ coincide
if and only if $\varphi$ is symmetric on average with respect to
$\Psi$. If $\varphi$ is asymptotically symmetric with respect to
$\Psi$ (see Definition \ref{def:asymptotically-symmetric}) or weakly
symmetric (cf. \cite{Stadlbauer11}) and the group is amenable, then
$\mathcal{P}\left(\varphi\circ\pi_{1},\sigma\rtimes\Psi\right)=\mathcal{P}\left(\varphi,\sigma\right)$
by Theorem \ref{thm:amenablesymmetric-implies-fullpressure} (cf.
\cite[Theorem 4.1]{Stadlbauer11}). Hence, a recurrent and asymptotically
symmetric potential is symmetric on average. For an amenable group
and a potential which is not necessarily recurrent, a necessary condition
for the base pressure and skew product pressure to coincide is not
known. A sufficient condition is that the potential is asymptotically
symmetric or weakly symmetric. It seems that the proofs given in \cite{Jaerisch12c}
or \cite[Theorem 4.1]{Stadlbauer11}, which make use of Kesten's classical
criterion for amenability, cannot be generalised to potentials which
are symmetric on average.

\subsection{Recurrent groups\label{sub:Outlook}}

A classical result by Pólya (\cite{MR1512028}) states that the simple
random walk on $\Z^{d}$ is recurrent if and only if $d\le2$. Dudley
(\cite{MR0141167}) has shown that a countable abelian group is recurrent
if it has rank at most two, where a \emph{recurrent group} is a group
which carries an irreducible recurrent random walk. Kesten (\cite{MR0214137})
motivated the interesting task of classifying recurrent groups and
it became known as Kesten's conjecture that finitely generated recurrent
groups are finite or contain $\Z$ or $\Z^{2}$ as a finite index
subgroup. Kesten's conjecture was proved by Varopoulos (\cite{MR832044})
using a famous result of Gromov (\cite{MR623534}). An extension for
random walks on locally finite graphs with quasi-transitive automorphism
groups is due to Woess (\cite[Corollary 4.7]{MR1246471},\cite[Theorem 4.1]{MR1397471}).
Another related result is due to Rees (\cite{MR627791,MR661820})
which shows the following for a non-trivial normal subgroup $N$ of
a geometrically finite Fuchsian groups $\Gamma$ such that $\Gamma/N\simeq\Z^{d}$
for some $d\in\N$: Firstly, the exponents of convergence $\delta\left(N\right)$
and $\delta\left(\Gamma\right)$ coincide. Further, if $\Gamma$ contains
no parabolic elements, then the Fuchsian group $N$ is of divergence
type if and only if $d\le2$. If $\Gamma$ contains parabolic elements,
then the situation is more complicated, however, the implication that
a divergence type subgroup $N$ satisfies $d\le2$ remains true. In
fact, the results in \cite{MR627791} are deduced for general Gibbs
measures satisfying a certain symmetry condition. However, we remark
that, in contrast to the result of Matsuzaki and Yabuki and to Proposition
\ref{prop:fullpressure-symmetric-on-average}, it is \emph{a priori}
assumed in the work of Rees that $\Gamma/N\simeq\Z^{d}$, for some
$d\in\N$.

In our framework, the classification of recurrent groups implies the
following. Let $\Psi:I^{*}\rightarrow G$ denote a surjective semigroup
homomorphism and let $\varphi:I^{\N}\rightarrow\R$ be a potential
depending only on the first coordinate. If $\varphi\circ\pi_{1}$
is recurrent, then $G$ is a recurrent group. In particular, since
each recurrent group is amenable, we obtain from Kesten's classical
theorem that if $\varphi$ is symmetric on average with respect to
$\Psi$, then  $\mathcal{P}\left(\varphi\circ\pi_{1},\sigma\rtimes\Psi\right)=\mathcal{P}\left(\varphi,\sigma\right)$.
Regarding Example \ref{recurrent-non-symmetric-example}, we observe
that, the potential $\varphi\circ\pi_{1}$ is recurrent, $\Z$ is
a recurrent group, and $\mathcal{P}\left(\varphi\circ\pi_{1},\sigma\rtimes\Psi\right)<\mathcal{P}\left(\varphi,\sigma\right)$
whenever the potential is not symmetric on average with respect to
$\Psi$. 

If we replace the full shift $I^{\N}$ by an irreducible Markov shift
$\Sigma$ with finite alphabet $I$, and if we consider $\varphi:\Sigma\rightarrow\R$
depending only on the first coordinate, then we obtain the following
by investigating the group of graph automorphisms using a result of
Woess (\cite[Theorem 4.1]{MR1397471}).
\begin{prop}
\label{prop:classification-of-recurrentgroups-for-markov}Let $\left(\Sigma\times G,\sigma\rtimes\Psi\right)$
be an irreducible group-extended Markov system with $\card\left(I\right)<\infty$.
Suppose that $\varphi:\Sigma\rightarrow\R$ depends only on the first
coordinate. If $\varphi\circ\pi_{1}$ is recurrent, then $G$ is a
recurrent group. \end{prop}
\begin{conjecture}
Let $\left(\Sigma\times G,\sigma\rtimes\Psi\right)$ be an irreducible
group-extended Markov system with $\card\left(I\right)<\infty$. Suppose
that $\varphi:\Sigma\rightarrow\R$ is Hölder continuous. If $\varphi\circ\pi_{1}$
is recurrent, then $G$ is a recurrent group. 
\end{conjecture}
The assertion of the conjecture was proved for a large class of Hölder
continuous potentials on $\Sigma_{\F_{t}}$ with $t\ge2$ under the
additional assumption that $G$ is a free abelian group in \cite[Theorem 4.7]{MR627791}.
It also worth noting that, if the conjecture is true, then it provides
a significant strengthening of Corollary \ref{cor:recurrence-implies-amenable}
(\ref{enu:recurrent-implies-amenable}), stating that recurrence of
$\varphi\circ\pi_{1}$ implies that $G$ is amenable.
\begin{acknowledgement*}
The author thanks Professor Stadlbauer for fruitful discussion and
for bringing the result of Zimmer to the author's attention. The author
thanks Sara Munday for valuable comments. The author is grateful to
Professor Morita for his support and to Professor Shirai for inviting
him to the Dynamical Systems Seminar of Kyushu university. 
\end{acknowledgement*}

\section{Preliminaries\label{sec:Preliminaries}}

\subsection{Symbolic dynamics\label{sub:Symbolic-Dynamics}}

Throughout, the underlying state space for the symbolic thermodynamic
formalism will be a \emph{Markov shift $\Sigma$, }given by 
\[
\Sigma:=\left\{ \omega:=\left(\omega_{1},\omega_{2},\ldots\right)\in I^{\N}:\; a\left(\omega_{i},\omega_{i+1}\right)=1\mbox{ for all }i\in\N\right\} ,
\]
where $I\subset\N$ denotes a finite or countable infinite \emph{alphabet},
the matrix $A=\left(a\left(i,j\right)\right)\in\left\{ 0,1\right\} ^{I\times I}$
is the \emph{incidence matrix} and the \emph{left shift map} $\sigma:\Sigma\rightarrow\Sigma$
is defined by $\sigma\left(\left(\omega_{1},\omega_{2},\ldots\right)\right):=\left(\omega_{2},\omega_{3},\ldots\right)$.
The set of \emph{$A$-admissible words} of length $n\in\mathbb{N}$
is given by 
\[
\Sigma^{n}:=\left\{ \omega\in I^{n}:\,\, a\left(\omega_{i},\omega_{i+1}\right)=1,\,\mbox{for all }i\in\left\{ 1,\dots,n-1\right\} \right\} .
\]
The set of $A$-admissible words of arbitrary length is denoted by
$\Sigma^{*}:=\bigcup_{n\in\N}\Sigma^{n}$. Let us also define the
\emph{word length function} $\left|\cdot\right|:\,\Sigma^{*}\cup\Sigma\rightarrow\N\cup\left\{ \infty\right\} $,
where for $\omega\in\Sigma^{*}$ we set $\left|\omega\right|$ to
be the unique $n\in\N$ such that $\omega\in\Sigma^{n}$, and for
$\omega\in\Sigma$ we set $\left|\omega\right|:=\infty$. For each
$\omega\in\Sigma^{*}\cup\Sigma$ and $n\in\N$ with $n\le\left|\omega\right|$,
we define $\omega_{|n}:=\left(\omega_{1},\dots,\omega_{n}\right)$.
For $\omega,\tau\in\Sigma$, we set $\omega\wedge\tau$ to be the
longest common initial block of $\omega$ and $\tau$, that is, $\omega\wedge\tau:=\omega_{|l}$,
where $l:=\sup\left\{ n\in\N:\omega_{|n}=\tau_{|n}\right\} $. For
$n\in\N$ and $\omega\in\Sigma^{n}$, we let $\left[\omega\right]:=\left\{ \tau\in\Sigma:\tau_{|n}=\omega\right\} $
denote the \emph{cylindrical set} given by $\omega$.

If $\Sigma$ is the Markov shift with alphabet $I$ whose incidence
matrix consists entirely of $1$s, then we have that $\Sigma=I^{\N}$
and $\Sigma^{n}=I^{n}$ for all $n\in\N$. Then we set $I^{*}:=\Sigma^{*}$.
For $\omega,\tau\in I^{*}$ we denote by $\omega\tau\in I^{*}$ the
\emph{concatenation} of $\omega$ and $\tau$, which is defined by
$\omega\tau:=\left(\omega_{1},\dots,\omega_{\left|\omega\right|},\tau_{1},\dots,\tau_{\left|\tau\right|}\right)$
for $\omega,\tau\in I^{*}$. Note that $I^{*}$ forms a semigroup
with respect to the concatenation operation. The semigroup $I^{*}$
is the free semigroup over the set $I$ and satisfies the following
universal property: For each semigroup $S$ and for every map $u:I\rightarrow S$,
there exists a unique semigroup homomorphism $\hat{u}:I^{*}\rightarrow S$
such that $\hat{u}\left(i\right)=u\left(i\right)$, for all $i\in I$
(see \cite[Section 3.10]{MR1650275}).

We equip $I^{\N}$ with the product topology of the discrete topology
on $I$. The Markov shift $\Sigma\subset I^{\N}$ is equipped with
the subspace topology. A countable basis of this topology on $\Sigma$
is given by the cylindrical sets $\left\{ \left[\omega\right]:\omega\in\Sigma^{*}\right\} $.
We will make use of the following metric generating the topology on
$\Sigma$. For $\alpha>0$, we define the metric $d_{\alpha}$ on
$\Sigma$ given by 
\[
d_{\alpha}\left(\omega,\tau\right):=\e^{-\alpha\left|\omega\wedge\tau\right|},\mbox{ for all }\omega,\tau\in\Sigma.
\]

For a function $f:\Sigma\rightarrow\R$ and $n\in\N$, we use the
notation $S_{n}f:\Sigma\rightarrow\R$ to denote the \emph{ergodic
sum} of $f$ with respect to the left shift $\sigma$, in other words,
$S_{n}f:=\sum_{i=0}^{n-1}f\circ\sigma^{i}$. Also, we set $S_{0}f:=0$. 

We say that $f:\Sigma\rightarrow\R$ is \emph{$\alpha$-H\"older
continuous}, for some $\alpha>0$, if 
\[
V_{\alpha}\left(f\right):=\sup_{n\ge1}\left\{ V_{\alpha,n}\left(f\right)\right\} <\infty,
\]
where for each $n\in\N$ we let 
\[
V_{\alpha,n}\left(f\right):=\sup\left\{ \frac{\left|f\left(\omega\right)-f\left(\tau\right)\right|}{d_{\alpha}\left(\omega,\tau\right)}:\omega,\tau\in\Sigma,\left|\omega\wedge\tau\right|\ge n\right\} .
\]
We say that \emph{$f$ }is\emph{ H\"older continuous} if there exists
$\alpha>0$ such that $f$ is $\alpha$-Hölder continuous.

The following fact is well-known (see e.g. \cite[Lemma 2.3.1]{MR2003772}). 
\begin{fact}
[Bounded distortion property] \label{fact-bounded-distortion-property}Let
$\Sigma$ be a Markov shift. If $f:\Sigma\rightarrow\R$ is Hölder
continuous, then there exists a constant $C_{f}\ge1$ such that, for
all $\omega\in\Sigma^{*}$ and $\tau,\tau'\in\left[\omega\right]$,
we have 
\[
\left|S_{\left|\omega\right|}f\left(\tau\right)-S_{\left|\omega\right|}f\left(\tau'\right)\right|\le\log C_{f}.
\]

\end{fact}
We will make use of the following topological mixing properties for
Markov shifts. 

\begin{defn} \label{mixing-definitions}Let $\Sigma$ be a Markov
shift with alphabet $I\subset\N$. 
\begin{itemize}
\item $\Sigma$ is \emph{irreducible} if, for all $i,j\in I$, there exists
$\omega\in\Sigma^{*}$ such that $i\omega j\in\Sigma^{*}$.
\item $\Sigma$ is \emph{topologically mixing} if, for all $i,j\in I$,
there exists $n_{0}\in\N$ with the property that, for all $n\ge n_{0}$,
there exists $\omega\in\Sigma^{n}$ such that $i\omega j\in\Sigma^{*}$.
\item $\Sigma$ is \emph{finitely primitive} if there exists $l\in\N$ and
a finite set $\Lambda\subset\Sigma^{l}$ with the property that, for
all $i,j\in I$, there exists $\omega\in\Lambda$ such that $i\omega j\in\Sigma^{*}$.
 
\end{itemize}
\end{defn} 
\begin{rem*}
$\Sigma$ is finitely primitive if and only if $\Sigma$ is topologically
mixing and if $\Sigma$ satisfies the big images and preimages property
(see \cite{MR1955261}). 
\end{rem*}
The following definition of the Gurevi\v{c} pressure was introduced
by Sarig (\cite[Definition 1]{MR1738951}) and extends the notion
of the Gurevi\v{c} entropy (\cite{MR0263162,MR0268356}) corresponding
to $\varphi=0$. 

\begin{defn} Let $\Sigma$ be a Markov shift with alphabet $I$ and
left shift $\sigma:\Sigma\rightarrow\Sigma$. Let $f:\Sigma\rightarrow\R$
be Hölder continuous. For each $a\in I$ and $n\in\N$, we set 
\[
Z_{n}\left(f,a,\sigma\right):=\sum_{x\in\Sigma,x_{1}=a,\sigma^{n}\left(x\right)=x}\e^{S_{n}f\left(x\right)}\quad\text{and}\quad Z_{n}^{*}\left(f,a,\sigma\right):=\sum_{x\in\Sigma,x_{1}=a,\sigma^{n}\left(x\right)=x,x_{2},\dots,x_{n}\neq a}\e^{S_{n}f\left(x\right)}.
\]
If $\Sigma$ is irreducible, then the \emph{Gurevi\v{c} pressure}
of $f$ with respect to $\sigma$ is for each $a\in I$ given by 
\[
\mathcal{P}\left(f,\sigma\right):=\limsup_{n\rightarrow\infty}\frac{1}{n}\log Z_{n}\left(f,a,\sigma\right).
\]
If $\Sigma$ is reducible, then we define 
\[
\mathcal{P}\left(f,\sigma\right):=\sup_{V\in\mathcal{V}}\mathcal{P}\left(f\big{|}_{V},\sigma\big{|}_{V}\right),
\]
where $\mathcal{V}$ denotes the set of irreducible components of
$\Sigma$.

\end{defn} 
\begin{rem}
[\cite{JaerischKessebohmer10}, Fact 3.1] Let $\Sigma$ be a Markov
shift with alphabet $I$ and left shift $\sigma:\Sigma\rightarrow\Sigma$.
If $\Sigma$ is irreducible and $f$ is Hölder continuous, then $\mathcal{P}\left(f,\sigma\right)$
is independent of $a\in I$. 
\end{rem}

The next definition goes back to the work of Ruelle and Bowen (cf.
\cite{MR0289084}, \cite{bowenequilibriumMR0442989}).
\begin{defn}
[Gibbs measure] \label{gibbs-measure}Let $f:\Sigma\rightarrow\R$
be Hölder continuous with $\mathcal{P}\left(f,\sigma\right)<\infty$.
We say that a Borel probability measure \emph{$\mu$ }is a\emph{ Gibbs
measure for $f$ }if there exists a constant $C>0$ such that 
\begin{equation}
C^{-1}\le\frac{\mu\left(\left[\omega\right]\right)}{\e^{S_{\left|\omega\right|}f\left(\tau\right)-\left|\omega\right|\mathcal{P}\left(f,\sigma\right)}}\le C,\mbox{ for all }\omega\in\Sigma^{*}\mbox{ and }\tau\in\left[\omega\right].\label{eq:gibbs-equation}
\end{equation}

\end{defn}
The following necessary and sufficient condition for the existence
of Gibbs measures is taken from \cite[Theorem 1]{MR1955261}. The
uniqueness part also follows from \cite[Theorem 2.2.4]{MR2003772}.
The existence of a $\sigma$-invariant Gibbs measure on a finitely
primitive Markov shift follows from \cite{MR2003772}.
\begin{thm}
[Existence of Gibbs measures] \label{thm:existence-of-gibbs-measures}Let
$\Sigma$ be topologically mixing and let $f:\Sigma\rightarrow\R$
be Hölder continuous. Then there exists a $\sigma$-invariant Gibbs
measure for $f$ if and only if $\Sigma$ is finitely primitive\textup{
and $\mathcal{P}\left(f,\sigma\right)<\infty$. If a }$\sigma$-invariant
Gibbs measure for $f$ exists, then it is unique. 
\end{thm}

\subsection{Recurrent potentials\label{sub:Recurrent-potentials}}

In this subsection we recall the definition of a recurrent potential,
which was introduced by Sarig for Hölder continuous potentials on
a topologically mixing countable state Markov shift (\cite[Definition 1]{MR1818392}).
This notion generalises the notion of recurrence for Markov chains
and the definition of $R$-recurrence for positive matrices by Vere-Jones
(\cite{MR0141160}). In this paper we adapt the definition of recurrence
to Hölder continuous potentials on an irreducible Markov shift, which
is not necessarily topologically mixing. One can easily verify that
also in this case, the notion of recurrence is independent of the
choice of $a\in I$. 
\begin{defn}
 \label{definition-recurrence-pos-null}Let $\Sigma$ be an irreducible
Markov shift with alphabet $I$ and left shift $\sigma:\Sigma\rightarrow\Sigma$.
Let $f:\Sigma\rightarrow\R$ be Hölder continuous with $\mathcal{P}\left(f,\sigma\right)<\infty$.
We say that 
\[
f\mbox{ is \emph{recurrent} if }\sum_{n\in\N}\e^{-n\mathcal{P}\left(f,\sigma\right)}Z_{n}\left(f,a,\sigma\right)=\infty,\mbox{ for some (hence all) }a\in I,
\]
and we say that $f$ is \emph{transient}, otherwise. Moreover, if
$f$ is recurrent, then we say that 
\[
f\mbox{ is \emph{positive recurrent }if }\sum_{n\in\N}n\e^{-n\mathcal{P}\left(f,\sigma\right)}Z_{n}^{*}\left(f,a,\sigma\right)<\infty,\mbox{ for some (hence all) }a\in I,
\]
\[
f\mbox{ is \emph{null recurrent} if }\sum_{n\in\N}n\e^{-n\mathcal{P}\left(f,\sigma\right)}Z_{n}^{*}\left(f,a,\sigma\right)=\infty,\mbox{ for some (hence all) }a\in I.
\]

\end{defn}
If an irreducible Markov shift $\Sigma$ has period $p\in\N$, then
it is convenient to study the $p$-th iterate of the dynamics. The
next remark shows how this is reflected in pressure and recurrence.
\begin{rem}
\label{periodicity-aperiodic-remark}Let $\Sigma$ be an irreducible
Markov shift with alphabet $I$, incidence matrix $A=\left(a\left(i,j\right)\right)$
and left shift $\sigma:\Sigma\rightarrow\Sigma$. If $\Sigma$ has
period $p\in\N$, then we consider $\sigma^{p}:\Sigma\rightarrow\Sigma$,
which is conjugated to the left shift $\sigma^{\left(p\right)}:\Sigma^{\left(p\right)}\rightarrow\Sigma^{\left(p\right)}$
on the Markov shift $\Sigma^{\left(p\right)}$ via the canonical bijection
$\iota:\Sigma^{\left(p\right)}\rightarrow\Sigma$, where $\Sigma^{\left(p\right)}$
is the Markov shift with the alphabet $\Sigma^{p}$, for which $\omega,\tau\in\Sigma^{p}$
are admissible if $a\left(\omega_{p},\tau_{1}\right)=1$. Let $f^{\left(p\right)}:\Sigma^{\left(p\right)}\rightarrow\R$
be given by $f^{\left(p\right)}:=S_{p}f\circ\iota$ and denote the
irreducible components of $\Sigma^{\left(p\right)}$ by $\mathcal{V}$.
One easily verifies that, for each $V\in\mathcal{V}$, 
\[
\mathcal{P}\left(f^{\left(p\right)},\sigma^{\left(p\right)}\right)=\mathcal{P}\left(f^{\left(p\right)}\big|_{V},\sigma^{\left(p\right)}\big|_{V}\right)=p\mathcal{P}\left(f,\sigma\right)
\]
and that $f^{\left(p\right)}\big|_{V}$ is (positive) recurrent if
and only if $f$ is (positive) recurrent. 
\end{rem}
The following was proved by Sarig (\cite[Theorem 1,  Remark 2, Proposition 1, Proposition 3]{MR1818392}).
\begin{thm}
\label{thm:sarig-recurrence-characterisation} Let $\Sigma$ be topologically
mixing Markov shift and let $f:\Sigma\rightarrow\R$ be Hölder continuous
with $\mathcal{P}\left(f,\sigma\right)<\infty$. Then $f$ is recurrent
if and only if there exists $\rho>0$ and a conservative measure $\nu$
on $\Sigma$, positive and finite on cylindrical sets, which satisfies
$\mathcal{L}_{f}^{*}\left(\nu\right)=\rho\nu$. In this case, $\log\rho=\mathcal{P}\left(f,\sigma\right)$
and there exists a continuous function $h:\Sigma\rightarrow\R^{+}$,
such that $\mathcal{L}_{f}\left(h\right)=\rho h$ and such that $\log h$
and $\log h\circ\sigma$ are Hölder continuous. Moreover, $\nu$ is
the unique measure (up to a constant multiple), which is fixed by
$\e^{-\mathcal{P}\left(f,\sigma\right)}\mathcal{L}_{f}^{*}$ and which
is positive and finite on cylindrical sets, and $h$ is the unique
positive continuous function (up to a constant multiple), which is
fixed by $\e^{-\mathcal{P}\left(f,\sigma\right)}\mathcal{L}_{f}$
and which is bounded on cylindrical sets. Furthermore, we have that
$f$ is positive recurrent if and only if $\int h\, d\nu<\infty$. \end{thm}
\begin{rem}
The uniqueness of $h$ for a recurrent potential $f$, as stated in
Theorem \ref{thm:sarig-recurrence-characterisation}, follows because
$h$ defines the invariant and $\sigma$-finite measure $h\, d\nu$,
which is absolutely continuous with respect to the conservative ergodic
measure $\nu$. Therefore, $h$ is unique $\nu$-almost everywhere
(see, for example, \cite[Theorem 1.5.6]{MR1450400}). Since $\nu$
is positive on cylindrical sets and $h$ is continuous, uniqueness
of $h$ follows.
\end{rem}

\subsection{Group extensions and amenability\label{sub:Group-Extensions}}

In this section, we consider a group-extended Markov system $\left(\Sigma\times G,\sigma\rtimes\Psi\right)$,
for a Markov shift $\Sigma$ with alphabet $I\subset\N$, a countable
group $G$ and a semigroup homomorphism $\Psi:I^{*}\rightarrow G$.
Ergodic properties of group extensions given by locally compact abelian
groups have also been studied in \cite{MR1803461,MR1906436}.
\begin{rem}
\label{groupextension-is-Markovshift}$\left(\Sigma\times G,\sigma\rtimes\Psi\right)$
is conjugated to the Markov shift with state space 
\[
\left\{ \left(\left(\omega_{j},g_{j}\right)\right)\in\left(I\times G\right)^{\N}:\omega\in\Sigma\mbox{ and }g_{j}\Psi\left(\omega_{j}\right)=g_{j+1},\,\, j\in\N\right\} .
\]
We do not distinguish between this Markov shift and $\left(\Sigma\times G,\sigma\rtimes\Psi\right)$.
\end{rem}

The following definition is due to von Neumann (\cite{vonNeumann1929amenabledef}). 
\begin{defn}
A discrete group\emph{ $G$ }is\emph{ amenable} if there exists a
finitely additive probability measure $\nu$ on the power set of $G$,
such that $\nu\left(A\right)=\nu\left(g\left(A\right)\right)$, for
all $g\in G$ and $A\subset G$. 
\end{defn}
Let us now state some recent results on amenability for group extensions
of Markov shifts. We will refer to these results in Section \ref{sub:Amenable-groups}
in order to clarify the context of the results obtained in this paper.
The next theorem is due to Stadlbauer (\cite[Theorem 5.4]{Stadlbauer11}).
For locally constant potentials on a finite state Markov shift $\Sigma$,
a similar statement has been proved in \cite[Theorem 1.1]{Jaerisch11a}
using different methods. 
\begin{thm}
\label{thm:stadlbauer-nogap-implies-amenable}Let $\Sigma$ be finitely
primitive and let $\left(\Sigma\times G,\sigma\rtimes\Psi\right)$
be an irreducible\emph{ }group-extended Markov system. Let $\varphi:\Sigma\rightarrow\R$
be Hölder continuous with $\mathcal{P}\left(\varphi,\sigma\right)<\infty$.
If $\mathcal{P}\left(\varphi\circ\pi_{1},\sigma\rtimes\Psi\right)=\mathcal{P}\left(\varphi,\sigma\right)$,
then $G$ is amenable. 
\end{thm}
The next theorem, which provides a converse of the previous theorem,
is taken from \cite[Corollary 1.6]{Jaerisch12c}. Under slightly different
assumptions, the theorem was proved in \cite[Theorem 5.3.11]{JaerischDissertation11}
and independently, by Stadlbauer in \cite[Theorem 4.1]{Stadlbauer11}. 
\begin{defn}
\label{def:asymptotically-symmetric}We say that $\varphi$ is \emph{asymptotically
symmetric with respect to $\Psi$ }(\cite[Definition 1.3]{Jaerisch12c})
if there exist $n_{0}\in\N$ and sequences $\left(c_{n}\right)_{n\in\N}$
and $\left(N_{n}\right)_{n\in\N}$ with the property that $\lim_{n}\left(c_{n}\right)^{1/n}=1$,
$\lim_{n}n^{-1}N_{n}=0$ and such that, for each $g\in G$ and $n\ge n_{0}$,
we have 
\[
\sum_{\omega\in\Sigma^{n}:\Psi\left(\omega\right)=g}\e^{\sup S_{n}\varphi_{|\left[\omega\right]}}\le c_{n}\sum_{\omega\in\Sigma^{*}:\Psi\left(\omega\right)=g^{-1},\, n-N_{n}\le\left|\omega\right|\le n+N_{n}}\e^{\sup S_{\left|\omega\right|}\varphi_{|\left[\omega\right]}}.
\]
\end{defn}
\begin{thm}
\label{thm:amenablesymmetric-implies-fullpressure}Let $\Sigma$ be
finitely primitive and let $\left(\Sigma\times G,\sigma\rtimes\Psi\right)$
be an irreducible\emph{ }group-extended Markov system. Let $\varphi:\Sigma\rightarrow\R$
be Hölder continuous with $\mathcal{P}\left(\varphi,\sigma\right)<\infty$.
If $\varphi$ is asymptotically symmetric with respect to $\Psi$
and if $G$ is amenable, then \textup{$\mathcal{P}\left(\varphi\circ\pi_{1},\sigma\rtimes\Psi\right)=\mathcal{P}\left(\varphi\right).$
} 
\end{thm}

\section{Proof of the main results\label{sec:Proof-of-the}}

\subsection{Proof of Theorem $\ref{thm:main-theorem}$}

The following lemma, which shows that, for a recurrent potential $\varphi\circ\pi_{1}$,
the associated eigenfunction of $\mathcal{L}_{\varphi\circ\pi_{1}}$
has product structure, is crucial for the proof of the main theorem. 
\begin{lem}
\label{lem:existence-homomorphism}Let $\left(\Sigma\times G,\sigma\rtimes\Psi\right)$
be a topologically mixing group-extended Markov system. Let $\varphi:\Sigma\rightarrow\R$
be Hölder continuous such that $\mathcal{P}\left(\varphi\circ\pi_{1},\sigma\rtimes\Psi\right)<\infty$.
Suppose that $\varphi\circ\pi_{1}$ is recurrent and let $h:\Sigma\times G\rightarrow\R^{+}$
be the unique continuous function (up to a constant multiple), which
is fixed by $\e^{-\mathcal{P}\left(\varphi\circ\pi_{1},\sigma\rtimes\Psi\right)}\mathcal{L}_{\varphi\circ\pi_{1}}$
and which is bounded on cylindrical sets. Then there exists a unique
homomorphism $c:G\rightarrow\left(\R^{+},\cdot\right)$ and a continuous
function $h_{1}:\Sigma\rightarrow\R^{+}$, bounded on cylindrical
sets, such that $h=\left(h_{1}\circ\pi_{1}\right)\left(c\circ\pi_{2}\right)$.
\end{lem}
\begin{proof}
For each $g\in G$, let $g^{*}h:\Sigma\times G\rightarrow\R$ be given
by $\left(g^{*}h\right)\left(\omega,\gamma\right):=h\left(\omega,g\gamma\right)$,
for each $\left(\omega,\gamma\right)\in\Sigma\times G$. A short calculation
shows that, for all $\left(\omega,\gamma\right)\in\Sigma\times G$,
\begin{align*}
\mathcal{L}_{\varphi\circ\pi_{1}}\left(g^{*}h\right)\left(\omega,\gamma\right) & =\sum_{i\in I:i\omega\in\Sigma}\e^{\varphi\circ\pi_{1}\left(i\omega,\gamma\left(\Psi\left(i\right)\right)^{-1}\right)}h\left(i\omega,g\gamma\left(\Psi\left(i\right)\right)^{-1}\right)\\
 & =\sum_{i\in I:i\omega\in\Sigma}\e^{\varphi\left(i\omega\right)}h\left(i\omega,g\gamma\left(\Psi\left(i\right)\right)^{-1}\right)\\
 & =\mathcal{L}_{\varphi\circ\pi_{1}}\left(h\right)\left(\omega,g\gamma\right)=\e^{\mathcal{P}\left(\varphi\circ\pi_{1},\sigma\rtimes\Psi\right)}h\left(\omega,g\gamma\right).
\end{align*}
We have thus shown that $\mathcal{L}_{\varphi\circ\pi_{1}}\left(g^{*}h\right)=\e^{\mathcal{P}\left(\varphi\circ\pi_{1},\sigma\rtimes\Psi\right)}g^{*}h$,
for each $g\in G$. By Theorem \ref{thm:sarig-recurrence-characterisation},
there exists a constant $c\left(g\right)\in\R^{+}$ such that $g^{*}h=c\left(g\right)h$.
This defines a homomorphism $c:G\rightarrow\left(\R^{+},\cdot\right)$,
since we have $c\left(g_{1}g_{2}\right)h=\left(g_{1}g_{2}\right)^{*}\left(h\right)=g_{2}^{*}\left(g_{1}^{*}\left(h\right)\right)=c\left(g_{1}\right)c\left(g_{2}\right)h$,
for all $g_{1},g_{2}\in G$. We conclude that there exists a continuous
function $h_{1}:\Sigma\rightarrow\R^{+}$, bounded on cylindrical
sets, such that $h=\left(h_{1}\circ\pi_{1}\right)\left(c\circ\pi_{2}\right)$. 
\end{proof}
We are now in the position to prove the main result. 
\begin{proof}
[Proof of Theorem $\ref{thm:main-theorem}$]First suppose that $\varphi\circ\pi_{1}$
is recurrent. Let $h:\Sigma\times G\rightarrow\R^{+}$ be the unique
continuous function (up to a constant multiple), which is fixed by
$\e^{-\mathcal{P}\left(\varphi\circ\pi_{1},\sigma\rtimes\Psi\right)}\mathcal{L}_{\varphi\circ\pi_{1}}$
and which is bounded on cylindrical sets. By Lemma \ref{lem:existence-homomorphism},
there exists a unique homomorphism $c:G\rightarrow\left(\R^{+},\cdot\right)$
and a continuous function $h_{1}:\Sigma\rightarrow\R^{+}$, bounded
on cylindrical sets, such that $h=\left(h_{1}\circ\pi_{1}\right)\left(c\circ\pi_{2}\right)$.
 Hence, the coboundary $\log h-\log h\circ\left(\sigma\rtimes\Psi\right)$
is given by 
\[
\log h-\log h\circ\left(\sigma\rtimes\Psi\right)=\log\left(h_{1}\circ\pi_{1}\right)-\log\left(h_{1}\circ\pi_{1}\right)\circ\left(\sigma\rtimes\Psi\right)+\log\left(c\circ\pi_{2}\right)-\log\left(c\circ\pi_{2}\right)\circ\left(\sigma\rtimes\Psi\right).
\]
A short calculation shows that, for each $\left(x,g\right)\in\Sigma\times G$,
\begin{equation}
\log\left(c\circ\pi_{2}\right)\left(x,g\right)-\log\left(c\circ\pi_{2}\right)\left(\sigma x,g\Psi\left(x_{1}\right)\right)=-\log c\left(\Psi\left(x_{1}\right)\right).\label{eq:00a}
\end{equation}
Therefore, the coboundary $\log h-\log h\circ\left(\sigma\rtimes\Psi\right)$
is independent of the second coordinate. Let $\varphi_{c}:\Sigma\rightarrow\R$
be given by $\varphi_{c}\left(x\right):=\varphi\left(x\right)-\log c\left(\Psi\left(x_{1}\right)\right)$,
for each $x\in\Sigma$. Clearly, $\varphi_{c}$ is Hölder continuous.
For ease of notation, let us also define $\tilde{\varphi}:\Sigma\rightarrow\R^{+}$
given by $\tilde{\varphi}:=\varphi_{c}+\log h_{1}-\log h_{1}\circ\sigma$.
We have 
\begin{equation}
\tilde{\varphi}\circ\pi_{1}=\varphi\circ\pi_{1}+\log h-\log h\circ\left(\sigma\rtimes\Psi\right).\label{eq:0a}
\end{equation}
Since $\log h$ and $\log h\circ\left(\sigma\rtimes\Psi\right)$ are
Hölder continuous by Theorem \ref{thm:sarig-recurrence-characterisation},
so are $\log h_{1}$, $\log h_{1}\circ\sigma$ and $\tilde{\varphi}$.
As a consequence, we have 
\begin{equation}
\mathcal{P}\left(\tilde{\varphi}\circ\pi_{1},\sigma\rtimes\Psi\right)=\mathcal{P}\left(\varphi\circ\pi_{1},\sigma\rtimes\Psi\right)\mbox{ and }\mathcal{P}\left(\tilde{\varphi},\sigma\right)=\mathcal{P}\left(\varphi_{c},\sigma\right).\label{eq:0}
\end{equation}
Since $h$ satisfies $\mathcal{L}_{\varphi\circ\pi_{1}}\left(h\right)=\e^{\mathcal{P}\left(\varphi\circ\pi_{1},\sigma\rtimes\Psi\right)}h$,
we have $\mathcal{L}_{\tilde{\varphi}\circ\pi_{1}}\left(\1\circ\pi_{1}\right)=\e^{\mathcal{P}\left(\varphi\circ\pi_{1},\sigma\rtimes\Psi\right)}\left(\1\circ\pi_{1}\right)$
by (\ref{eq:0a}). Combining this with the identity $\mathcal{L}_{\tilde{\varphi}\circ\pi_{1}}\left(\1\circ\pi_{1}\right)=\left(\mathcal{L}_{\tilde{\varphi}}\left(\1\right)\right)\circ\pi_{1}$,
we conclude that 
\begin{equation}
\mathcal{L}_{\tilde{\varphi}}\left(\1\right)=\e^{\mathcal{P}\left(\varphi\circ\pi_{1},\sigma\rtimes\Psi\right)}\1.\label{eq:1}
\end{equation}
In particular, we have $\mathcal{P}\left(\tilde{\varphi},\sigma\right)\le\mathcal{P}\left(\varphi\circ\pi_{1},\sigma\rtimes\Psi\right)<\infty$.
Since $\Sigma$ is finitely primitive and $\tilde{\varphi}$ is Hölder
continuous with $\mathcal{P}\left(\tilde{\varphi},\sigma\right)<\infty$,
we have that $\e^{-n\mathcal{P}\left(\tilde{\varphi},\sigma\right)}\mathcal{L}_{\tilde{\varphi}}^{n}\left(\1\right)$
converges uniformly to a bounded Hölder continuous function (see \cite[Theorem 2.4.6]{MR2003772}).
Since $\e^{-n\mathcal{P}\left(\tilde{\varphi},\sigma\right)}\mathcal{L}_{\tilde{\varphi}}^{n}\left(\1\right)=\e^{-n\left(\mathcal{P}\left(\tilde{\varphi},\sigma\right)-\mathcal{P}\left(\varphi\circ\pi_{1},\sigma\rtimes\Psi\right)\right)}\1$
by (\ref{eq:1}), we conclude that $\mathcal{P}\left(\tilde{\varphi},\sigma\right)=\mathcal{P}\left(\varphi\circ\pi_{1},\sigma\rtimes\Psi\right)$.
Combining with (\ref{eq:0}), we have thus shown that 
\begin{equation}
\mathcal{P}\left(\tilde{\varphi}\circ\pi_{1},\sigma\rtimes\Psi\right)=\mathcal{P}\left(\varphi\circ\pi_{1},\sigma\rtimes\Psi\right)=\mathcal{P}\left(\tilde{\varphi},\sigma\right)=\mathcal{P}\left(\varphi_{c},\sigma\right).\label{eq:2}
\end{equation}
Further, by definition of $\tilde{\varphi}$ and (\ref{eq:2}), the
equality in (\ref{eq:1}) implies 
\begin{equation}
\mathcal{L}_{\varphi_{c}}\left(h_{1}\right)=\e^{\mathcal{P}\left(\varphi_{c},\sigma\right)}h_{1}.\label{eq:3}
\end{equation}
Since $\Sigma$ is finitely primitive and $\varphi_{c}$ is Hölder
continuous with $\mathcal{P}\left(\varphi_{c},\sigma\right)<\infty$,
we have that $\varphi_{c}$ is positive recurrent by \cite[Corollary 2]{MR1955261}.
Since $h_{1}$ is continuous and bounded on cylindrical sets, it follows
from (\ref{eq:3}) and \cite[Corollary 2 and 3]{MR1955261} that $h_{1}$
is Hölder continuous and bounded away from zero and infinity. In particular,
$\tilde{\varphi}$ is cohomologous to $\varphi_{c}$ in the cohomology
class of bounded continuous functions, which implies that $\mu_{\tilde{\varphi}}=\mu_{\varphi_{c}}$
(see \cite[Theorem 2.2.7]{MR2003772}). We now show that $\mu_{\tilde{\varphi}}\times\lambda$
is conservative. By (\ref{eq:1}) and (\ref{eq:2}) we have $\mathcal{L}_{\tilde{\varphi}}^{*}\left(\mu_{\tilde{\varphi}}\right)=\e^{\mathcal{P}\left(\tilde{\varphi}\circ\pi_{1},\sigma\rtimes\Psi\right)}\mu_{\tilde{\varphi}}$,
which gives 
\begin{equation}
\mathcal{L}_{\tilde{\varphi}\circ\pi_{1}}^{*}\left(\mu_{\tilde{\varphi}}\times\lambda\right)=\e^{\mathcal{P}\left(\tilde{\varphi}\circ\pi_{1},\sigma\rtimes\Psi\right)}\left(\mu_{\tilde{\varphi}}\times\lambda\right).\label{eq:3a}
\end{equation}
 From (\ref{eq:0a}) and since $\varphi\circ\pi_{1}$ is recurrent,
we have that $\tilde{\varphi}\circ\pi_{1}$ is recurrent. Hence, (\ref{eq:3a})
implies that $\mu_{\tilde{\varphi}}\times\lambda$ is conservative
by Theorem \ref{thm:sarig-recurrence-characterisation}. 

We now turn to the proof of the converse implication and the uniqueness
of the homomorphism $c$ in Theorem \ref{thm:main-theorem}. To prove
this, suppose that $\mathcal{P}\left(\varphi_{c},\sigma\right)<\infty$
and that $\mu_{\varphi_{c}}\times\lambda$ is conservative, for some
homomorphism $c:G\rightarrow\left(\R^{+},\cdot\right)$. We show that
$\varphi\circ\pi_{1}$ is recurrent and that $c$ is unique. Since
$\Sigma$ is finitely primitive, there exists a Hölder continuous
function $h_{1}:\Sigma^{+}\rightarrow\R$, bounded away from zero
and infinity, such that $\mathcal{L}_{\varphi_{c}}\left(h_{1}\right)=\e^{\mathcal{P}\left(\varphi_{c},\sigma\right)}h_{1}$.
Setting $\tilde{\varphi}:=\varphi_{c}+\log h_{1}-\log h_{1}\circ\sigma$
as above, we have  $\mathcal{L}_{\tilde{\varphi}\circ\pi_{1}}^{*}\left(\mu_{\varphi_{c}}\times\lambda\right)=\e^{\mathcal{P}\left(\varphi_{c},\sigma\right)}\left(\mu_{\varphi_{c}}\times\lambda\right)$.
Since $\mu_{\varphi_{c}}\times\lambda$ is conservative, it follows
from Theorem \ref{thm:sarig-recurrence-characterisation} that $\tilde{\varphi}\circ\pi_{1}$
is recurrent and that $\mathcal{P}\left(\varphi_{c},\sigma\right)=\mathcal{P}\left(\tilde{\varphi}\circ\pi_{1},\sigma\rtimes\Psi\right)$.
Since $h_{1}$ is Hölder continuous, we conclude that $\varphi_{c}\circ\pi_{1}$
is recurrent and $\mathcal{P}\left(\tilde{\varphi}\circ\pi_{1},\sigma\rtimes\Psi\right)=\mathcal{P}\left(\varphi_{c}\circ\pi_{1},\sigma\rtimes\Psi\right)$.
Finally, the following relation, which is deduced from (\ref{eq:00a}),
\begin{equation}
\varphi_{c}\circ\pi_{1}=\varphi\circ\pi_{1}+\log\left(c\circ\pi_{2}\right)-\log\left(c\circ\pi_{2}\right)\circ\left(\sigma\rtimes\Psi\right),\label{eq:3b}
\end{equation}
gives that $\varphi\circ\pi_{1}$ is recurrent. To prove uniqueness
of $c$, observe that we have $\mathcal{P}\left(\varphi_{c}\circ\pi_{1},\sigma\rtimes\Psi\right)=\mathcal{P}\left(\varphi\circ\pi_{1},\sigma\rtimes\Psi\right)$
by (\ref{eq:3b}). We have thus shown that $\mathcal{P}\left(\varphi_{c},\sigma\right)=\mathcal{P}\left(\varphi\circ\pi_{1},\sigma\rtimes\Psi\right)$.
Therefore, we have $\mathcal{L}_{\varphi_{c}\circ\pi_{1}}\left(h_{1}\circ\pi_{1}\right)=\e^{\mathcal{P}\left(\varphi\circ\pi_{1},\sigma\rtimes\Psi\right)}\left(h_{1}\circ\pi_{1}\right)$.
Hence, $\mathcal{L}_{\varphi\circ\pi_{1}}\left(\left(c\circ\pi_{2}\right)\left(h_{1}\circ\pi_{1}\right)\right)=\e^{\mathcal{P}\left(\varphi\circ\pi_{1},\sigma\rtimes\Psi\right)}\left(c\circ\pi_{2}\right)\left(h_{1}\circ\pi_{1}\right)$
by (\ref{eq:3b}).  Since $\left(c\circ\pi_{2}\right)\left(h_{1}\circ\pi_{1}\right)$
is continuous and bounded on cylindrical sets, recurrence of $\varphi\circ\pi_{1}$
implies that $c$ is unique by Theorem \ref{thm:sarig-recurrence-characterisation}. 

To finish the proof, we are left to prove Theorem \ref{thm:main-theorem}
(\ref{enu:pressure-relation}), (\ref{enu:eigenfunction}), (\ref{enu:conformal-measure})
and (\ref{enu:equilibrium-measure}). The assertion in (\ref{enu:pressure-relation})
follows from (\ref{eq:2}). Claim (\ref{enu:eigenfunction}) follows
from Lemma \ref{lem:existence-homomorphism} and (\ref{eq:3}). To
prove (\ref{enu:conformal-measure}), let $f:\Sigma\rightarrow\R^{+}$.
By (\ref{eq:3b}) we have 
\begin{equation}
\mathcal{L}_{\varphi\circ\pi_{1}}\left(f\right)=\left(c\circ\pi_{2}\right)\mathcal{L}_{\varphi_{c}\circ\pi_{1}}\left(\frac{f}{c\circ\pi_{2}}\right).\label{eq:4}
\end{equation}
Let $\nu_{1}$ denote the unique Borel probability measure, such that
$\mathcal{L}_{\varphi_{c}}^{*}\left(\nu_{1}\right)=\e^{\mathcal{P}\left(\varphi_{c},\sigma\right)}\nu_{1}$.
By (\ref{eq:2}) we have 
\begin{equation}
\mathcal{L}_{\varphi_{c}\circ\pi_{1}}^{*}\left(\nu_{1}\times\lambda\right)=\e^{\mathcal{P}\left(\varphi\circ\pi_{1},\sigma\rtimes\Psi\right)}\left(\nu_{1}\times\lambda\right).\label{eq:5}
\end{equation}
We verify that $\mathcal{L}_{\varphi\circ\pi_{1}}^{*}\left(\left(c\circ\pi_{2}\right)^{-1}\, d\left(\nu_{1}\times\lambda\right)\right)=\e^{\mathcal{P}\left(\varphi\circ\pi_{1},\sigma\rtimes\Psi\right)}\left(c\circ\pi_{2}\right)^{-1}\, d\left(\nu_{1}\times\lambda\right)$.
By using (\ref{eq:4}) and (\ref{eq:5}), we have 
\[
\mathcal{L}_{\varphi\circ\pi_{1}}^{*}\left(\frac{d\left(\nu_{1}\times\lambda\right)}{c\circ\pi_{2}}\right)\left(f\right)=\frac{d\left(\nu_{1}\times\lambda\right)}{c\circ\pi_{2}}\left(\left(c\circ\pi_{2}\right)\mathcal{L}_{\varphi_{c}\circ\pi_{1}}\left(\frac{f}{c\circ\pi_{2}}\right)\right)=\e^{\mathcal{P}\left(\varphi\circ\pi_{1},\sigma\rtimes\Psi\right)}\left(\nu_{1}\times\lambda\right)\left(\frac{f}{c\circ\pi_{2}}\right).
\]
This finishes the proof of (\ref{enu:conformal-measure}). Clearly,
the assertion in (\ref{enu:equilibrium-measure}) follows by combining
(\ref{enu:eigenfunction}),  (\ref{enu:conformal-measure}) and the
fact that $\mu_{\varphi_{c}}=h_{1}\, d\nu_{1}$. The proof is complete.
\end{proof}

\subsection{Proof of Proposition \ref{prop:fullpressure-symmetric-on-average}
and Remark \ref{irreducible-remark}}
\begin{proof}
[Proof of Proposition $\ref{prop:fullpressure-symmetric-on-average}$]Let
$\varphi\circ\pi_{1}$ be recurrent and let $c:G\rightarrow\left(\R^{+},\cdot\right)$
denote the homomorphism given by Theorem \ref{thm:main-theorem}.
By Theorem \ref{thm:main-theorem} (\ref{enu:equilibrium-measure}),
we have that $\mu_{\varphi_{c}}\times\lambda$ is the equilibrium
measure of $\varphi\circ\pi_{1}$. In particular, $\left(\Sigma\times G,\mu_{\varphi_{c}}\times\lambda,\sigma\rtimes\Psi\right)$
is a conservative ergodic measure preserving dynamical system with
transfer operator $\e^{-\mathcal{P}\left(\varphi\circ\pi_{1},\sigma\rtimes\Psi\right)}h^{-1}\mathcal{L}_{\varphi\circ\pi_{1}}\left(h\cdot\right)$,
where $h=\left(h_{1}\circ\pi_{1}\right)\left(c\circ\pi_{2}\right)$,
for the Hölder continuous function $h_{1}:\Sigma\rightarrow\R^{+}$,
bounded away from zero and infinity, given by Theorem \ref{thm:main-theorem}
(\ref{enu:eigenfunction}). By the ratio ergodic theorem (see e.g.
\cite[Theorem 2.2.1]{MR1450400}), we conclude that, for each $g\in G$,
we have $\left(\mu_{\varphi_{c}}\times\lambda\right)$-almost everywhere
\[
\lim_{n\rightarrow\infty}\frac{\sum_{k=1}^{n}\e^{-k\mathcal{P}\left(\varphi\circ\pi_{1},\sigma\rtimes\Psi\right)}h^{-1}\mathcal{L}_{\varphi\circ\pi_{1}}^{k}\left(h\1_{\Sigma\times\left\{ g\right\} }\right)}{\sum_{k=1}^{n}\e^{-k\mathcal{P}\left(\varphi\circ\pi_{1},\sigma\rtimes\Psi\right)}h^{-1}\mathcal{L}_{\varphi\circ\pi_{1}}^{k}\left(h\1_{\Sigma\times\{g^{-1}\}}\right)}=\frac{\left(\mu_{\varphi_{c}}\times\lambda\right)\left(\1_{\Sigma\times\left\{ g\right\} }\right)}{\left(\mu_{\varphi_{c}}\times\lambda\right)\left(\1_{\Sigma\times\{g^{-1}\}}\right)}=1.
\]
Consequently, using that, for each $k\in\N$ and $g\in G$, 
\[
\left(\sup h_{1}\right)^{-1}c\left(g\right)^{-1}\mathcal{L}_{\varphi\circ\pi_{1}}^{k}\left(h\1_{\Sigma\times\left\{ g\right\} }\right)\le\mathcal{L}_{\varphi\circ\pi_{1}}^{k}\left(\1_{\Sigma\times\left\{ g\right\} }\right)\le\left(\inf h_{1}\right)^{-1}c\left(g\right)^{-1}\mathcal{L}_{\varphi\circ\pi_{1}}^{k}\left(h\1_{\Sigma\times\left\{ g\right\} }\right),
\]
it follows that we have $\left(\mu_{\varphi_{c}}\times\lambda\right)$-almost
everywhere 
\begin{equation}
c\left(g\right)^{-2}\frac{\inf h_{1}}{\sup h_{1}}\le\limsup_{n\rightarrow\infty}\frac{\sum_{k=1}^{n}\e^{-k\mathcal{P}\left(\varphi\circ\pi_{1},\sigma\rtimes\Psi\right)}\mathcal{L}_{\varphi\circ\pi_{1}}^{k}\left(\1_{\Sigma\times\left\{ g\right\} }\right)}{\sum_{k=1}^{n}\e^{-k\mathcal{P}\left(\varphi\circ\pi_{1},\sigma\rtimes\Psi\right)}\mathcal{L}_{\varphi\circ\pi_{1}}^{k}\left(\1_{\Sigma\times\{g^{-1}\}}\right)}\le c\left(g\right)^{-2}\frac{\sup h_{1}}{\inf h_{1}}.\label{eq:ratioergodic-estimate0}
\end{equation}
Since $\Sigma$ is finitely primitive, there exists a finite set
$F\subset I$ such that, for each $a\in I$, there exists $i\in F$
with $ai\in\Sigma^{2}$. Fix $F$ and choose a point $x\left(i\right)\in\left[i\right]$,
for each $i\in F$. Because $(\mu_{\varphi_{c}}\times\lambda)$ is
positive on cylindrical sets and the Hölder continuous function $\varphi$
has the bounded distortion property by Fact \ref{fact-bounded-distortion-property},
it follows from (\ref{eq:ratioergodic-estimate0}) that there exists
$C_{\varphi}\ge1$, such that for each $g\in G$,  
\begin{equation}
c\left(g\right)^{-2}\frac{\inf h_{1}}{\sup h_{1}}C_{\varphi}^{-2}\le\limsup_{n\rightarrow\infty}\frac{\sum_{k=1}^{n}\e^{-k\mathcal{P}\left(\varphi\circ\pi_{1},\sigma\rtimes\Psi\right)}\sum_{i\in F}\mathcal{L}_{\varphi\circ\pi_{1}}^{k}\left(\1_{\Sigma\times\left\{ g\right\} }\right)\left(x\left(i\right),\id\right)}{\sum_{k=1}^{n}\e^{-k\mathcal{P}\left(\varphi\circ\pi_{1},\sigma\rtimes\Psi\right)}\sum_{i\in F}\mathcal{L}_{\varphi\circ\pi_{1}}^{k}\left(\1_{\Sigma\times\{g^{-1}\}}\right)\left(x\left(i\right),\id\right)}\le c\left(g\right)^{-2}\frac{\sup h_{1}}{\inf h_{1}}C_{\varphi}^{2}.\label{eq:ratioergodic-estimate2}
\end{equation}
A homomorphism $c:G\rightarrow\left(\R^{+},\cdot\right)$ is bounded
away from zero (bounded away from infinity, respectively) if and only
if $c=1$. Hence, by (\ref{eq:ratioergodic-estimate2}), we obtain
that $c=1$ if and only if 
\begin{equation}
\sup_{g\in G}\limsup_{n\rightarrow\infty}\frac{\sum_{k=1}^{n}\e^{-k\mathcal{P}\left(\varphi\circ\pi_{1},\sigma\rtimes\Psi\right)}\sum_{i\in F}\mathcal{L}_{\varphi\circ\pi_{1}}^{k}\left(\1_{\Sigma\times\left\{ g\right\} }\right)\left(x(i),\id\right)}{\sum_{k=1}^{n}\e^{-k\mathcal{P}\left(\varphi\circ\pi_{1},\sigma\rtimes\Psi\right)}\sum_{i\in F}\mathcal{L}_{\varphi\circ\pi_{1}}^{k}\left(\1_{\Sigma\times\{g^{-1}\}}\right)\left(x(i),\id\right)}<\infty.\label{eq:ratioergodic-estimate3}
\end{equation}
To finish the proof, we show that (\ref{eq:ratioergodic-estimate3})
holds if and only if $\varphi$ is symmetric on average with respect
to $\Psi$. Let $k\in\N$ and $g\in G$. Since $\card\left(F\right)<\infty$,
we have that 
\begin{equation}
\sum_{i\in F}\sum_{\omega\in\Sigma^{k}:\Psi\left(\omega\right)=g^{-1},\omega x(i)\in\Sigma}\e^{S_{k}\varphi\left(\omega x\left(i\right)\right)}\le\card\left(F\right)\sum_{\omega\in\Sigma^{k}:\Psi\left(\omega\right)=g^{-1}}\e^{\sup S_{k}\varphi_{|\left[\omega\right]}}.\label{eq:pf-00}
\end{equation}
Because for each $\omega\in\Sigma^{*}$ there exists $i\in F$, such
that $\omega x(i)\in\Sigma$ and since $\varphi$ has the bounded
distortion property, we have 
\begin{equation}
\sum_{\omega\in\Sigma^{k}:\Psi\left(\omega\right)=g^{-1}}\e^{\sup S_{k}\varphi_{|\left[\omega\right]}}\le C_{\varphi}\sum_{i\in F}\sum_{\omega\in\Sigma^{k}:\Psi\left(\omega\right)=g^{-1},\omega x(i)\in\Sigma}\e^{S_{k}\varphi\left(\omega x\left(i\right)\right)}.\label{eq:pf-01}
\end{equation}
Moreover, observe that, for each $x\in\Sigma$, we have 
\begin{equation}
\mathcal{L}_{\varphi\circ\pi_{1}}^{k}\left(\1_{\Sigma\times\left\{ g\right\} }\right)\left(x,\id\right)=\sum_{\omega\in\Sigma^{k}:\Psi\left(\omega\right)=g^{-1},\omega x\in\Sigma}\e^{S_{k}\varphi\left(\omega x\right)}.\label{eq:pf-1}
\end{equation}
Combining (\ref{eq:pf-00}), (\ref{eq:pf-01}) and (\ref{eq:pf-1}),
we obtain that 
\[
C_{\varphi}^{-1}\sum_{\omega\in\Sigma^{k}:\Psi\left(\omega\right)=g^{-1}}\e^{\sup S_{k}\varphi_{|\left[\omega\right]}}\le\sum_{i\in F}\mathcal{L}_{\varphi\circ\pi_{1}}^{k}\left(\1_{\Sigma\times\left\{ g\right\} }\right)\left(x(i),\id\right)\le\card\left(F\right)\sum_{\omega\in\Sigma^{k}:\Psi\left(\omega\right)=g^{-1}}\e^{\sup S_{k}\varphi_{|\left[\omega\right]}}.
\]
Hence, (\ref{eq:ratioergodic-estimate3}) holds if and only if $\varphi$
is symmetric on average with respect to $\Psi$.
\end{proof}

\begin{proof}
[Proof of Remark   $\ref{irreducible-remark}$] Let $p>1$ denote
the period of the irreducible Markov shift $\left(\Sigma\times G,\sigma\rtimes\Psi\right)$.
We consider the left shift $\sigma^{\left(p\right)}:\Sigma^{\left(p\right)}\rightarrow\Sigma^{\left(p\right)}$
on the Markov shift $(\Sigma^{\left(p\right)},\sigma^{\left(p\right)})$,
which is conjugated to $\left(\Sigma,\sigma^{p}\right)$ via $\iota:\Sigma^{\left(p\right)}\rightarrow\Sigma$
as in Remark \ref{periodicity-aperiodic-remark}. Let $\Psi^{\left(p\right)}:(\Sigma^{\left(p\right)})^{*}\rightarrow G$
denote the unique semigroup homomorphism such that $\Psi^{\left(p\right)}\left(\omega\right)=\Psi\left(\omega\right)$,
for each $\omega\in\Sigma^{p}$, and observe that $\left(\sigma\rtimes\Psi\right)^{p}:\Sigma\times G\rightarrow\Sigma\times G$
is conjugated to the aperiodic Markov shift $\sigma^{\left(p\right)}\rtimes\Psi^{\left(p\right)}:\Sigma^{\left(p\right)}\times G\rightarrow\Sigma^{\left(p\right)}\times G$.
Denote the set of irreducible components of $\big(\Sigma^{\left(p\right)}\times G,\sigma^{\left(p\right)}\rtimes\Psi^{\left(p\right)}\big)$
by $\mathcal{V}$. Using that $\Sigma$ is finitely primitive and
$\Sigma\times G$ is irreducible, one verifies that there exists $l\in\N$
such that
\begin{equation}
\forall i,j\in I\,\exists\tau\in\Sigma^{pl}:\, i\tau j\in\Sigma^{pl+2}\mbox{ and }\Psi\left(\tau\right)=\id.\label{eq:fin-prim-wrt-loops}
\end{equation}
From (\ref{eq:fin-prim-wrt-loops}), it follows that there exists
a partition $\left(G_{V}\right)_{V\in\mathcal{V}}$ of $G$ such that
$V=\Sigma^{\left(p\right)}\times G_{V}$, for each $V\in\mathcal{V}$.
We choose the component $V_{\id}$ containing $\Sigma^{\left(p\right)}\times\left\{ \id\right\} $
and note that $V_{\id}=\Sigma^{\left(p\right)}\times G_{0}$ where
$G_{0}:=\bigcup_{n\in\N}\Psi\left(\Sigma^{np}\right)$. Using (\ref{eq:fin-prim-wrt-loops})
one obtains that $G_{0}$ is a subsemigroup of $G$. To prove that
$G_{0}$ is a group, let $g\in G_{0}$, $n\in\N$ and $\omega\in\Sigma^{np}$
such that $g=\Psi\left(\omega\right)$. Since $\Sigma\times G$ is
irreducible, there exists $r\in\N$ and $\gamma\in\Sigma^{r}$ such
that $\overline{\omega\gamma}:=\left(\omega\gamma\omega\gamma\dots\right)\in\Sigma$
and $\Psi\left(\omega\gamma\right)=\id$. Since $(\sigma\rtimes\Psi)^{np+r}\left(\overline{\omega\gamma},\id\right)=\left(\overline{\omega\gamma},\id\right)$
and $(\Sigma\times G,\sigma\rtimes\Psi)$ has period $p$, there exists
$m\in\N$ such that $r=mp$. Hence, we have $g^{-1}=\Psi\left(\gamma\right)\in G_{0}$.
Let us also verify that the index of $G_{0}$ in $G$ is at most $p$.
To prove this, fix some $\alpha\in\Sigma^{p}$ and let $g\in G$.
Since $\left(\Sigma\times G,\sigma\rtimes\Psi\right)$ is irreducible,
there exists $\tau\in\Sigma^{*}$ such that $g=\Psi\left(\alpha\right)\Psi\left(\tau\right)$.
Then there exists $u\in\left\{ 0,\dots,p-1\right\} $ such that $g\in\Psi\left(\alpha_{1}\dots\alpha_{u}\right)G_{0}$,
which proves that the index is at most $p$. 

Next, define the potential $\varphi^{\left(p\right)}:\Sigma^{\left(p\right)}\rightarrow\R$,
given by $\varphi^{\left(p\right)}:=S_{p}\varphi\circ\iota$. As observed
in Remark \ref{periodicity-aperiodic-remark}, we have 
\begin{equation}
\mathcal{P}\left(\varphi^{\left(p\right)}\circ\pi_{1}\big|_{V_{\id}},\sigma^{\left(p\right)}\rtimes\Psi^{\left(p\right)}\big|_{V_{\id}}\right)=p\mathcal{P}\left(\varphi\circ\pi_{1},\sigma\rtimes\Psi\right),\quad\mathcal{P}\left(\varphi^{\left(p\right)},\sigma^{\left(p\right)}\right)=p\mathcal{P}\left(\varphi,\sigma\right).\label{eq:mainproof-irreducible-1}
\end{equation}
Further, since $\varphi\circ\pi_{1}$ is recurrent, so is $\varphi^{\left(p\right)}\circ\pi_{1}\big|_{V_{\id}}$.
Moreover, $\Sigma^{\left(p\right)}$ is finitely primitive and $\sigma^{\left(p\right)}\rtimes\Psi^{\left(p\right)}:\Sigma^{\left(p\right)}\times G_{0}\rightarrow\Sigma^{\left(p\right)}\times G_{0}$
is topologically mixing. Hence, the results of Corollary \ref{cor:recurrence-implies-amenable}
hold with $G$ replaced by $G_{0}$. Since the index of $G_{0}$
in $G$ is finite, it is well-known that $G$ is amenable if and only
if $G_{0}$ is amenable. Also, $G$ is finite if and only if $G_{0}$
is finite. Hence, the results of Corollary \ref{cor:recurrence-implies-amenable}
hold for $G$ under the assumption that $\Sigma\times G$ is irreducible. 

Next we prove that the equality $\mathcal{P}\left(\varphi\circ\pi_{1},\sigma\rtimes\Psi\right)=\mathcal{P}\left(\varphi,\sigma\right)$
holds if and only if 
\begin{equation}
\sup_{g\in G_{0}}\limsup_{n\rightarrow\infty}\frac{\sum_{k=1}^{n}\e^{-kp\mathcal{P}\left(\varphi\circ\pi_{1},\sigma\rtimes\Psi\right)}\sum_{\omega\in\Sigma^{kp}:\Psi\left(\omega\right)=g}\e^{\sup S_{kp}\varphi_{|\left[\omega\right]}}}{\sum_{k=1}^{n}\e^{-kp\mathcal{P}\left(\varphi\circ\pi_{1},\sigma\rtimes\Psi\right)}\sum_{\omega\in\Sigma^{kp}:\Psi\left(\omega\right)=g^{-1}}\e^{\sup S_{kp}\varphi_{|\left[\omega\right]}}}<\infty.\label{eq:phi-p-symmetriconaverage}
\end{equation}
It follows from (\ref{eq:mainproof-irreducible-1}) and Proposition
\ref{prop:fullpressure-symmetric-on-average}, applied to the group-extended
Markov system $(\Sigma^{\left(p\right)}\times G_{0},\sigma^{\left(p\right)}\rtimes\Psi^{\left(p\right)})$
and the potential $\varphi^{\left(p\right)}$, that $\mathcal{P}\left(\varphi\circ\pi_{1},\sigma\rtimes\Psi\right)=\mathcal{P}\left(\varphi,\sigma\right)$
holds if and only if $\varphi^{\left(p\right)}$ is symmetric on average
with respect to $\Psi^{\left(p\right)}$. It is easy to see that $\varphi^{\left(p\right)}$
is symmetric on average with respect to $\Psi^{\left(p\right)}$ if
(\ref{eq:phi-p-symmetriconaverage}) holds. The proof is complete.
\end{proof}

\subsection{Proof of Proposition \ref{prop:classification-of-recurrentgroups-for-markov}}
\begin{proof}
[Proof of Proposition $\ref{prop:classification-of-recurrentgroups-for-markov}$]
Without loss of generality, we may assume that $\card\left(G\right)=\infty$.
Using that $\varphi\circ\pi_{1}$ is recurrent, similarly as in \cite[Proof of Theorem 1.2]{Jaerisch11a},
we deduce that there exists a recurrent random walk $P$ on the undirected
graph $X$ with vertex set $V:=I\times G$, in which $\left(i,g\right),\left(j,h\right)\in I\times G$
are connected by an edge if 
\[
\left(\sigma\rtimes\Psi\right)^{-1}\left(\left[i\right]\times\left\{ g\right\} \right)\cap\left(\left[j\right]\times\left\{ h\right\} \right)\neq\emptyset\,\,\,\mbox{or}\,\,\left(\sigma\rtimes\Psi\right)^{-1}\left(\left[j\right]\times\left\{ h\right\} \right)\cap\left(\left[i\right]\times\left\{ g\right\} \right)\neq\emptyset.
\]
Since $\left(\Sigma\times G,\sigma\rtimes\Psi\right)$ is irreducible,
we have that $X$ is connected. Let $d$ denote the associated graph
metric of $X$. Further, let $\mathrm{Aut}\left(X,P\right)$ denote
the group of all invertible self-isometries $\gamma$ of the metric
space $\left(X,d\right)$ which satisfy $p\left(x,y\right)=p\left(\gamma x,\gamma y\right)$,
for all $x,y\in V$. Following \cite[Proof of Theorem 1.2]{Jaerisch11a},
each element $g\in G$ gives rise to a self-isometry $\gamma_{g}$
on $\left(X,d\right)$ given by $\gamma_{g}\left(i,h\right):=\left(i,gh\right)$,
for all $\left(i,h\right)\in I\times G$. Since $\varphi\circ\pi_{1}$
does not depend on the second coordinate, one immediately verifies
that $\gamma_{g}\in\mathrm{Aut}\left(X,P\right)$, for each $g\in G$.
Note that $\iota:G\rightarrow\mathrm{Aut}\left(X,P\right)$ given
by $g\mapsto\gamma_{g}$ defines a monomorphism of groups and we may
thus consider $G$ as a subgroup of $\mathrm{Aut}\left(X,P\right)$.
We equip $\mathrm{Aut}\left(X,P\right)$ with the topology of pointwise
convergence and note that $\iota\left(G\right)$ is a  discrete subgroup
of $\mathrm{Aut}\left(X,P\right)$. We do not distinguish between
$G$ and $\iota\left(G\right)$. Since $\card\left(I\right)<\infty$,
 the action of $G$ is quasi-transitive, that is, $G$ acts with finitely
many orbits, namely the orbits are $\left\{ i\right\} \times G$,
$i\in I$. Since the random walk $P$ is recurrent, it follows from
\cite[Theorem 5.13]{MR1743100} that $X$ is roughly isometric with
the $1$- or $2$-dimensional grid. In particular, the growth function
of $X$ is polynomial of degree one or two. We can then apply \cite[Theorem 4.1]{MR1397471}
to the discrete subgroup $G$ of $\mathrm{Aut}\left(X,P\right)$.
The theorem states that there exists a compact normal subgroup $N$
of $G$ such that $G/N$ contains $\Z$ or $\Z^{2}$ as a finite index
subgroup. Because $N$ is compact and discrete, we have that $N$
is finite and hence, $G$ is a recurrent group. 
\end{proof}
\providecommand{\bysame}{\leavevmode\hbox to3em{\hrulefill}\thinspace}
\providecommand{\MR}{\relax\ifhmode\unskip\space\fi MR }
% \MRhref is called by the amsart/book/proc definition of \MR.
\providecommand{\MRhref}[2]{%
  \href{http://www.ams.org/mathscinet-getitem?mr=#1}{#2}
}
\providecommand{\href}[2]{#2}

\end{document}